\pdfoutput=1
\documentclass[a4paper,12pt]{amsart}

\usepackage{amsthm}
\usepackage{amsmath}
\usepackage{amssymb}
\usepackage{latexsym}
\usepackage{enumerate}

\def\cprime{$'$}
\usepackage{graphicx}

\newtheorem{theoremA}{Theorem}

\renewcommand{\thetheoremName}

\newtheorem{theorem}{Theorem}[section]
\newtheorem{lemma}[theorem]{Lemma}

\newtheorem{proposition}[theorem]{Proposition}
\newtheorem{corollary}[theorem]{Corollary}

\theoremstyle{definition}
\newtheorem{definition}[theorem]{Definition}
\newtheorem{example}[theorem]{Example}

\newtheorem{remark}[theorem]{Remark}

\numberwithin{equation}{section}

\newcommand{\Hess}{\operatorname{Hess}}
\newcommand{\dist}{\operatorname{dist}}
\newcommand{\Vol}{\operatorname{Vol}}

\newcommand{\Div}{\operatorname{div}}
\newcommand{\C}{\operatorname{Cap}}

\newcommand{\erre}{\mathbb{R}}

\newcommand{\tN}{\widetilde{N}}

\newcommand{\LL}{\operatorname{L}}
\newcommand{\Len}{\operatorname{Length}}
\newcommand{\A}{\operatorname{Area}}

\begin{document}

\title[Isoperimetric Analysis]{Extrinsic Isoperimetric Analysis \\ on Submanifolds \\ with Curvatures bounded from below}

\author[S. Markvorsen]{Steen Markvorsen$^{\#}$}
\address{Department of Mathematics, Technical University of Denmark.}
%\curraddr{}
\email{S.Markvorsen@mat.dtu.dk}
\thanks{$^{\#}$ Work partially supported by
the Danish Natural Science Research Council and DGI grant
MTM2004-06015-C02-02.}

\author[V. Palmer]{Vicente Palmer*}
\address{Departament de Matem\`{a}tiques, Universitat Jaume I, Castellon,
Spain.}
\email{palmer@mat.uji.es}
\thanks{* Work partially supported by
the Caixa Castell\'{o} Foundation, DGI grant MTM2004-06015-C02-02,
and by the Danish Natural Science Research Council.}

%    General info
%\date{January 1, 1994 and, in revised form, June 22, 1994.}

%\dedicatory{This paper is dedicated to [[[]]].}


%%%%%%%%%%%%%%%%%%%%%%%%%%%%%%%%%%%%%%%%%%%%%%%%%%%%%%%%%%%%%%%%%%%%%%%%
%%%%%%%%%%%%%%%%%%%%%%%%%%%%%%%%%%%%%%%%%%%%%%%%%%%%%%%%%%%%%%%%%%%%%%%%
%%%%%%%%%%%%%%%%%%%%%%%%%%%%%%%%%%%%%%%%%%%%%%%%%%%%%%%%%%%%%%%%%%%%%%%%
%%%%%%%%%%%%%%%%%%%%%%%%%%%%%%%%%%%%%%%%%%%%%%%%%%%%%%%%%%%%%%%%%%%%%%%%
%%%%%%%%%%%%%%%%%%%%%%%%%%%%%%%%%%%%%%%%%%%%%%%%%%%%%%%%%%%%%%%%%%%%%%%%
%%%%%%%%%%%%%%%%%%%%%%%%%%%%%%%%%%%%%%%%%%%%%%%%%%%%%%%%%%%%%%%%%%%%%%%%
%%%%%%%%%%%%%%%%%%%%%%%%%%%%%%%%%%%%%%%%%%%%%%%%%%%%%%%%%%%%%%%%%%%%%%%%
%%%%%%%%%%%%%%%%%%%%%%%%%%%%%%%%%%%%%%%%%%%%%%%%%%%%%%%%%%%%%%%%%%%%%%%%

\begin{abstract}
We obtain upper bounds for the isoperimetric quotients of extrinsic balls of submanifolds
in ambient spaces which have a lower bound on their {\em{radial sectional curvatures}}.
The submanifolds are themselves only assumed to have lower bounds on
the radial part of the mean curvature vector field and on the radial part
of the intrinsic unit normals at the boundaries of
the extrinsic spheres, respectively. In the same vein
we also establish {\em{lower bounds}} on the {\em{mean exit time}} for Brownian motions in the
extrinsic balls, i.e. lower bounds for the time it takes (on average) for Brownian particles
to diffuse {\em{within}} the extrinsic ball from a given starting point before they hit the
boundary of the extrinsic ball. In those cases, where we may extend our analysis
to hold all the way to infinity, we apply a {\em{capacity comparison}} technique to obtain a sufficient
condition for the submanifolds to be {\em{parabolic}}, i.e. a condition which will guarantee that
any Brownian particle, which is free to move around in the whole submanifold, is bound to eventually
revisit any given neighborhood
of its starting point with probability $1$.
The results of this paper are in a rough sense {\em{dual}} to similar results obtained
previously by the present authors in complementary settings where we assume that the curvatures
are bounded from {\em{above}}.
\end{abstract}

%%%%%%%%%%%%%%%%%%%%%%%%%%%%%%%%%%%%%%%%%%%%%%%%%%%%%%%%%%%%%%%%%%%%%%%%%%
%%%%%%%%%%%%%%%%%%%%%%%%%%%%%%%%%%%%%%%%%%%%%%%%%%%%%%%%%%%%%%%%%%%%%%%%%%

\maketitle

%%%%%%%%%%%%%%%%%%%%%%%%%%%%%%%%%%%%%%%%%%%%%%%%%%%%%%%%%%%%%%%%%%%%%%%%%%
%%%%%%%%%%%%%%%%%%%%%%%%%%%%%%%%%%%%%%%%%%%%%%%%%%%%%%%%%%%%%%%%%%%%%%%%%%

\section{Introduction}  \label{secIntro}
Given a precompact domain $\Omega$ in a Riemannian manifold $M$, the isoperimetric
quotient for $\Omega$ measures
the ratio between the volume of the boundary and the volume of the
enclosed domain: $\mathcal{Q}(\Omega) \, = \, \Vol(\partial \Omega)/\Vol(\Omega)$.
These volume measures and this quotient
are descriptors of fundamental importance for obtaining geometric and a\-na\-ly\-tic information
about the manifold $M$. In fact, this holds true on every zoom level, be it global, local, or micro-local.

\subsection{On the global level} \label{subsecGlobal} A classical quest is
to find necessary and sufficient conditions for the {\em{type}} of a given manifold:
Is it hyperbolic or parabolic? As already alluded to in the abstract, parabolicity is a first measure of the relative
smallness of the boundary at infinity of the manifold: The Brownian particles are bound to eventually
return to any given neighborhood of their starting point - as in $\mathbb{R}^{2}$; they do not get lost at infinity as they do in $\mathbb{R}^{3}$ (which the simplest example of a transient manifold). \\
In \cite{LS} T. Lyons and D. Sullivan collected and proved a
number of equivalent conditions (the so-called Kelvin--Nevanlinna--Royden criteria) for non-parabolicity, i.e. hyperbolicity:
The Riemannian manifold $\,(M, g)\,$ is hyperbolic if one (thence all) of the following equivalent conditions are satisfied: (a) $M$ has {\em finite resistance} to infinity, (b)
$M$ has {\em positive capacity}, (c)
$M$ admits a {\em{Green's function}}, (d)
There exists a precompact open domain $\Omega$, such that
the Brownian motion starting from $\Omega$ does
{\em{not return}} to $\Omega$ with probability $1$, (e)
$M$ admits a {\em{square integrable vector field}} with finite, but non-zero, absolute divergence. In particular the capacity condition (b) implies that  $\,(M, g)\,$ is parabolic if it has vanishing capacity - a condition which we will apply in section
\ref{secCapAnalysis}. \\
Returning now to the role of isoperimetric information: J. L. Fernandez showed in \cite{F}
that $M$ is hyperbolic if the so-called (rooted) isoperimetric profile function  $\phi(t)\,$  has a quare integrable
reciprocal, i.e. $\,\int^{\Vol(M)}\phi^{-2}(t)\,dt \, < \, \infty\,$. Here the $\Omega_{0}-$rooted profile function is molded directly from isoperimetric information as follows:
\begin{equation*}
\begin{aligned}
\phi(t) \, = \,  \operatorname{inf}
 \{&\Vol(\partial\Omega)\,
\, | \\
 \, \, &\Omega \, \, \text{is a smooth relatively compact domain in }
 M \, \, ,\\
       &\Omega \supset \Omega_{0}\, \, , \quad \text{and} \, \,
        \operatorname{Vol}(\Omega) \, \geq \, t \, \, \} \quad .
\end{aligned}
\end{equation*}
The volume of a non-parabolic manifold $M$ is necessarily infinite.
In fact, a finite volume manifold is parabolic by the following theorem due
to Grigor'yan, Karp, Lyons and Sullivan, and Varopoulos.
See \cite{Gri} section 7.2 for an account of this type of
results, which again is stated in terms of the simplest possible 'isoperimetric' information:
Let $B_{r}(q)$ denote the geodesic ball centered at $q$ in
$M$ and with radius $r$.  If there exists a  point $q$ such that one (or both) of the
following conditions is satisfied
\begin{equation*}
\begin{aligned}
\int^{\infty} \frac{r}{\Vol(B_{r}(q))} \, dr  \, &= \, \infty \\
\int^{\infty} \frac{1}{\Vol(\partial B_{r}(q))} \, dr  \, &= \, \infty \quad ,
\end{aligned}
\end{equation*}
then $M$ is parabolic.\\
In the present paper we obtain generalizations of this parabolicity condition. They are obtained for submanifolds in ambient spaces with a {\em{lower bound}} on curvatures by using a capacity comparison technique in combination with the Kelvin--Nevanlinna--Royden condition (b) as stated above, see Proposition \ref{propBallCond}
and Theorem \ref{thmSphereCondGen}.
These results complement - and are in a rough sense dual to - previous hyperbolicity results that we have obtained using a corresponding {\em{upper bound}} on the curvatures of the ambient spaces, see \cite{MP2, MP3}.

\subsection{On the local level} If the boundary of a given domain is relatively small as compared
to the domain itself, we also expect the mean exit time for Brownian motion to be
correspondingly larger. The main concern of the present paper is to show an upper bound on the isoperimetric quotients
of the so-called extrinsic balls of submanifolds under the essential assumption that the ambient spaces have
a {\em{lower bound}} on their curvatures, see Theorem \ref{thmIsopGeneral}. The result for the mean exit time (from such extrinsic balls) then follows, as observed and proved in Theorem \ref{thmExitTime}.
These results are again dual to results which have been previously obtained under
the condition of an upper bound on the curvature of the ambient space, see  \cite{Pa2, MP1} and \cite{Ma1, Pa1}\,,  respectively.

\subsection{On the micro-local level} When considering again the intrinsic geodesic balls $B_{\varepsilon}(q)$ of $M$ centered at
a fixed point $q$ and assuming that the radius
$\varepsilon$ is approaching $0$,  then the Taylor series expansion of the
volume function of the corresponding
metric ball (or metric sphere) contains information about the
curvatures of $M$ at $q$ - a classical result (for surfaces) obtained by Gauss and
developed by A. Gray and L. Vanhecke in their seminal work \cite{MR521460}. Moreover,
the geodesic metric balls have been analyzed by A. Gray and M. Pinsky in \cite{GraP} in order to extract the
geometric information contained in that particular  function of $\varepsilon\,$, which gives the mean exit time from the center point $q$ of the metric ball $B_{\varepsilon}(q)$, see Theorem \ref{thmGraP} below in section \ref{secExitTime}. \\
To motivate even further the {\em{extrinsic geometric}} setting under consideration in the present paper, we mention here also yet another nice observation due to L. Karp and M. Pinsky concerning submanifolds in $\mathbb{R}^{n}\,$; see \cite{KP, KPextrin}\,, where they show how to extract combinations of the principal curvatures of a submanifold at a point $q$ from suitable power series expansions of the respective functions $\Vol(D_{\varepsilon}(q))$, $\Vol(\partial D_{\varepsilon}(q))$, and $E_{D_{\varepsilon}(q)}$, where $D_{\varepsilon}(q)$ is the extrinsic ball of (extrinsic) radius $\varepsilon$ centered at $q$, and $E$ is the mean exit time function.\\

On all levels then, be it global, local, or micro-local, as well as from both viewpoints, intrinsic or extrinsic, we thus encounter a fundamental interplay and inter-de\-pen\-den\-ce between the highly instrumental geometric concepts of measure, shape, and diffusion considered here, namely the notions of volume, curvature, and exit time, respectively. \\

%%%%%%%%%%%%%%%%%%%%%%%%%%%%%%%%%%%%%%%%%%%%%%%%%%%%%%%%%%%%%%%%%%%%%%%%%%
%%%%%%%%%%%%%%%%%%%%%%%%%%%%%%%%%%%%%%%%%%%%%%%%%%%%%%%%%%%%%%%%%%%%%%%%%%
%%%%%%%%%%%%%%%%%%%%%%%%%%%%%%%%%%%%%%%%%%%%%%%%%%%%%%%%%%%%%%%%%%%%%%%%%%
%%%%%%%%%%%%%%%%%%%%%%%%%%%%%%%%%%%%%%%%%%%%%%%%%%%%%%%%%%%%%%%%%%%%%%%%%%

\subsection{Outline of the paper} \label{subsecOutline}
In the first three sections \ref{secFirstGlimpse}, \ref{secPrevRes}, and  \ref{secExampSurfOfRev},
we first provide intuitive versions, shadows, of our main results, i.e. we
present  the general results under stronger conditions than actually needed and compare them in particular with
previous techniques for obtaining  parabolicity (for surfaces of revolution) due to J. Milnor and K. Ichihara.
In section \ref{secGeoBounds} we then begin to establish the technical machinery for the paper and give
precise definitions of the geometric bounds needed as preparation for our definition of what we call an
Isoperimetric Constellation in section \ref{secIsopConstellation}. This key notion is then applied in section
\ref{secMainIsopRes} to present and prove our
main isoperimetric result, Theorem \ref{thmIsopGeneral}.
Various consequences of the main result and its proof are shown in sections \ref{secConseq} and \ref{secIntrinsic}.
In particular we find new inequalities involving the volumes of extrinsic balls
and their derivatives as well as
intrinsic versions of our main results.
In section \ref{secExitTime} we find a lower bound on the mean exit time function from extrinsic balls,
and establish in section \ref{secCapAnalysis} an inequality for the capacities of extrinsic annular domains,
which is then  finally applied to prove the parabolicity result, Theorem \ref{thmSphereCondGen}, as
alluded to above.

%%%%%%%%%%%%%%%%%%%%%%%%%%%%%%%%%%%%%%%%%%%%%%%%%%%%%%%%%%%%%%%%%%%%%%%%%%
%%%%%%%%%%%%%%%%%%%%%%%%%%%%%%%%%%%%%%%%%%%%%%%%%%%%%%%%%%%%%%%%%%%%%%%%%%
%%%%%%%%%%%%%%%%%%%%%%%%%%%%%%%%%%%%%%%%%%%%%%%%%%%%%%%%%%%%%%%%%%%%%%%%%%
%%%%%%%%%%%%%%%%%%%%%%%%%%%%%%%%%%%%%%%%%%%%%%%%%%%%%%%%%%%%%%%%%%%%%%%%%%

\section{A first Glimpse of the Main Results} \label{secFirstGlimpse}
We first facilitate intuition concerning our main results by considering some of their consequences for submanifolds in
constant curvature ambient spaces - in particular for surfaces in $\mathbb{R}^{3}$. This seems quite relevant and worthwhile,
because even in these strongly restricted settings we find results, which we believe are of independent interest. The results presented here are but shadows of the general results. The full versions of the main theorems
 appear in the sections below as indicated in the Outline, section \ref{subsecOutline}.\\

\subsection{Strong Assumptions and Constant Curvature} \label{subsecStrong}
The general strong conditions applied for these initial statements are as follows:
We let $P^{m}$ denote a complete immersed submanifold in an ambient
space form $N^{n} \, = \, \mathbb{K}^{n}_{b}$ with constant sectional curvature $b \leq 0$.
Suppose further that $P$ is {\em{radially mean $C-$convex}} in $N$ as viewed from a point $p$ in
the following sense: The unique oriented, arc length parametrized geodesic $\,\gamma_{p \to x}\,$ from $p$ to $x \in P$
in the ambient space $\mathbb{K}^{n}_{b}$ has an inner product with the mean curvature vector $H_{P}(x)$ of
$P$ in $N$ at $x$, which is bounded as follows:
\begin{equation}
\mathcal{C}(x)\, = \, -\langle\, \gamma'_{p \to x\,} \, , \, H_{P}(x)\, \rangle \, \geq \, C
\end{equation}
for some constant $C$.
This condition with $C \, = \, 0$ is e.g. satisfied by {\em{convex}} hypersurfaces, see remark \ref{remHypCon}, as well as by all {\em{minimal}} submanifolds.
We then consider a special type of compact subsets of $P$, the so-called extrinsic balls $D_{R}$, which for any given $R\, > \, 0$ consists of those
points $x$ in $P$, which have extrinsic distance to $p$ less than or equal to $R$, see Figure 1.\\

\begin{figure}[tb]
\includegraphics[width=40mm]{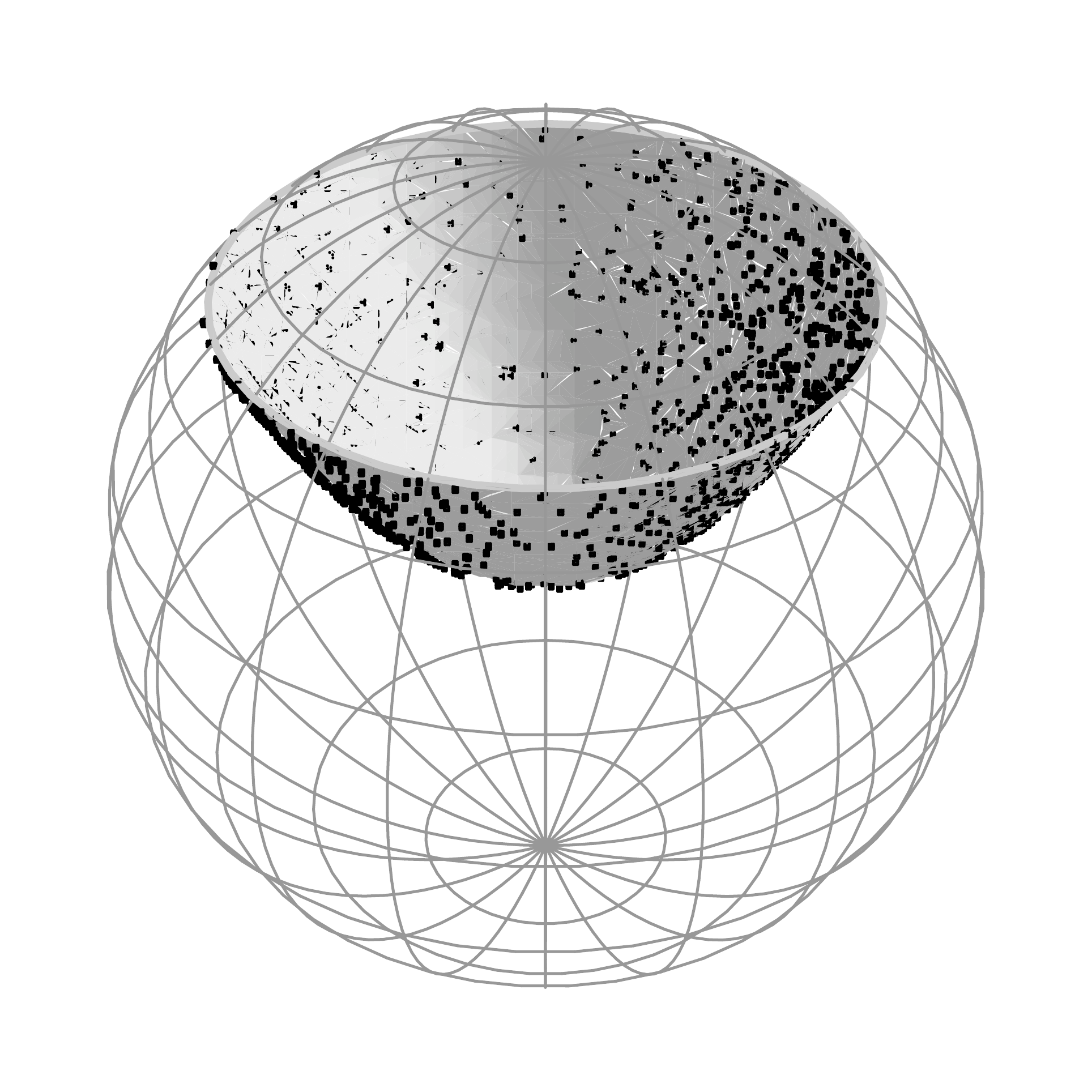} \quad
\includegraphics[width=40mm]{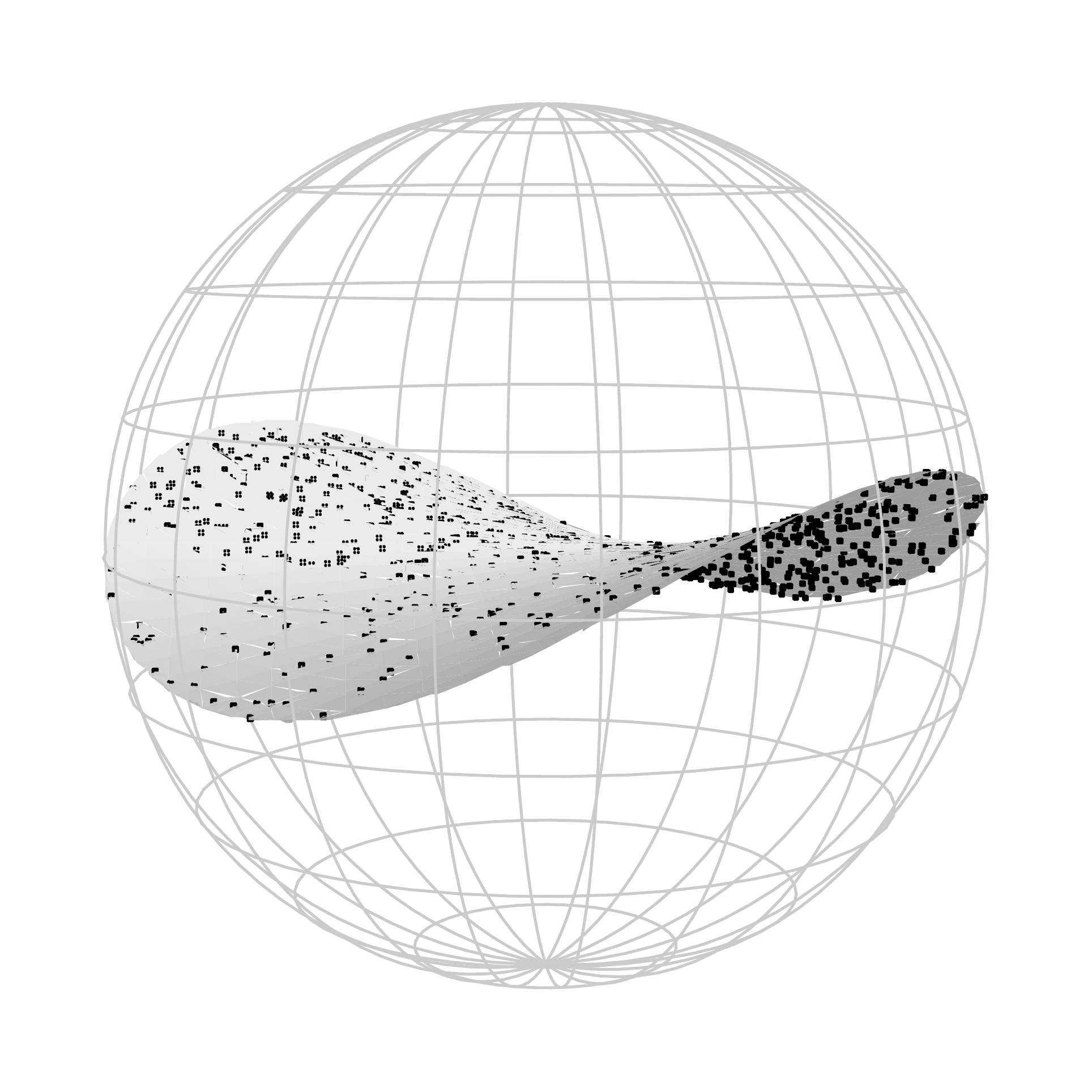}
\begin{center} \caption{Examples of extrinsic discs cut by a sphere in $\mathbb{R}^{3}$
from a paraboloid and from a minimal surface, respectively.
Both extrinsic discs are radially mean $0-$convex from the center point of the cutting sphere.
Theorem \ref{thmAisop} provides an upper bound on the isoperimetric quotients for
the respective extrinsic discs, cf. remark \ref{remSurfIsop}.}
\end{center} \label{FigChipConstruct}
\end{figure}

At each point $x$ on its boundary $\partial D_{R}$, the extrinsic ball $D_{R}$ has a unique
outward pointing unit vector $\nu_{x}\,$, which has the following inner product with the unit tangent vector of
$\,\gamma_{p \to x}\,$:
\begin{equation} \label{eqTangency}
\mathcal{T}(x)\, = \, \langle\, \gamma'_{p \to x\,} \, , \, \nu_{x}\, \rangle  \quad.
\end{equation}
We let $g(R)$ denote the minimum tangency value $\mathcal{T}$
along the boundary of the extrinsic ball
$D_{R}\,$ - in general:
\begin{equation} \label{eqTangencyBound}
g(r) \, = \, \min_{x \in \partial D_{r}} \mathcal{T}(x) \quad .
\end{equation}
An instrumental assumption to be satisfied throughout this paper, is that $g(r) \, > \, 0$ for all $r \,\leq R$
for each $D_{R}$ under consideration. In particular this then implies that the extrinsic balls always  have only one connected
compact component because any other compact component than the ones containing $p$ would automatically
have a point $q$ at minimum distance to $p$,
where then $\mathcal{T}(q) \, = \, 0$, which contradicts $g(r(q)) \, > \, 0\,$.  \\

With these initial ingredients and concepts we can now state the following
first instances of our main theorems:

\begin{theoremA}[The full version is stated in Theorem \ref{thmIsopGeneral}] \label{thmAisop}
If we as\-su\-me the strong conditions stated above in section \ref{subsecStrong}, then the following upper bound on
the isoperimetric quotients for extrinsic balls holds true for $\, C \, \leq \sqrt{-b}\,$:
\begin{equation} \label{eqIsopGeneralFirst}
\frac{\Vol(\partial D_{R})}{\Vol(D_{R})}\, \leq \,
\frac{m}{g(R)}\left(\eta_{Q_{b}}(R) - C\right) \quad ,
\end{equation}
where $\eta_{Q_{b}}(R)$ is the constant mean curvature of the metric hyper-sphere of radius $R$ in the space form of
constant curvature $b$ (see (\ref{eqHbversusWarp})); this function has $\, \lim_{r \to \infty}\,\eta_{Q_{b}}(r)\, = \, \sqrt{-b}\,$.
\end{theoremA}

\begin{remark} \label{remSurfIsop}
For the surfaces displayed in Figure 1 the
isoperimetric inequality reads as follows (since $C\,=\,0$ and $b\,=\,0\,$):
\begin{equation} \label{eqIsopGeneralFirst2D}
\frac{\Len(\partial D_{R})}{\A(D_{R})}\, \leq \,
\frac{2}{Rg(R)}\quad,
\end{equation}
where $g(R)$ is the lower bound of the tangency function $\mathcal{T}$ along the boundary curve $\partial D_{R}$ of the
extrinsic disc, see (\ref{eqTangency}) and (\ref{eqTangencyBound}).
\end{remark}

\begin{theoremA}[The full version is stated in Theorem \ref{thmExitTime}] \label{thmAexit}
We assume the strong general conditions stated above in section \ref{subsecStrong}.
Then the mean exit time $E_{R}(x)$, i.e. the time that it takes (on average)
for a Brownian particle starting from a point $x \in D_{R}$ to reach the boundary
$\partial D_{R}$ enjoys a lower bound of the following radial type:
\begin{equation}
E_{R}(x) \, \geq \, \mathcal{E}_{R}(r(x)) \quad ,
\end{equation}
where $\mathcal{E}_{R}(r)$ is a well defined radial function, which only depends on $R$, $b$, $C$, and $g(r)$.
The precise form of this dependence will, of course, be substantially explicated below,
see section \ref{secExitTime} and
Theorem \ref{thmExitTime} therein.
\end{theoremA}

\begin{theoremA}[The full version is stated in Theorem \ref{thmSphereCondGen}] \label{thmAparab}
We assume the strong general conditions stated above in section \ref{subsecStrong} for all radii $R$ of the extrinsic disc $D_{R}\,$.
Suppose further that
\begin{equation}\label{eqAparab}
\int^{\infty} Q_{b}(r)\exp\left(-\,\int_{1}^{r}\frac{m}{g^{2}(t)}\left(\eta_{Q_{b}}(t) - C \right)\, dt\right)\,dr\, = \, \infty \quad.
\end{equation}
Then $P^{m}$ is parabolic.
\end{theoremA}

In particular, we may, and do, extract the following first corollaries directly from the above
theorem.

\begin{corollary} \label{corAparab}
Let $P^{2}$ denote a two-dimensional radially mean $0-$ convex surface in Euclidean space $\mathbb{R}^{n}$ with
a radial tangency bounding function $g(r)$ which satisfies the following inequality for all sufficiently large $r$:
\begin{equation} \label{eqCorA}
g(r) \, \geq \, \tilde{g}(r) \, = \,  \sqrt{\frac{2\log(r)}{1 + 2\log(r)}} \quad .
\end{equation}
Then $P^{2}$ is parabolic.
\end{corollary}
\begin{remark} \label{remZero}
This particular result should be viewed in the light of the fact that there are well known
minimal (hence $0-$convex) surfaces in $\mathbb{R}^{3}$ - like Scherk's doubly periodic minimal surface -
which are non-parabolic (i.e. hyperbolic), but which nevertheless also - in partial contrast to what could be expected from the
above Corollary -
support radial tangency functions $\mathcal{T}(x)$ which are 'mostly' close to $1$ at infinity.
The Scherk surface alluded to is the graph surface of
the function $f(x,y)\, = \, \log\left(\cos(y)/\cos(x)\right)$, which is smooth and well-defined on a
checkerboard pattern in the $(x, y)-$plane.
It was proved in \cite{MMT} - via methods quite different from those considered in the
present paper - that Scherk's surface is hyperbolic. Roughly speaking,
the Scherk's surface may be considered
as the most 'slim' known hyperbolic surface in $\mathbb{R}^{3}$. If that surface is
viewed from a point far away from the $(x,y)-$plane,
the surface looks like two sets of parallel half-planes, both orthogonal to the $(x,y)-$plane,
one set below and the other above,
and the two sets being rotated $\pi/2$
with respect to each other, see e.g. \cite[pp. 46--49\,]{Mcrm}.

The radial tangency (from any fixed point $p$ in the $(x,y)-$plane)
is 'mostly' close to $1$ at infinity except
for the points in the $(x,y)-$plane itself, where the tangency function is 'wiggling'
sufficiently close to $0$, so that the condition \ref{eqCorA} cannot be satisfied.
In fact, if we accept for a moment the rough
description of the surface as two sets of parallel half planes,
then the integral over an extrinsic
disc $D_{R}$ of the tangency function $\mathcal{T}(x)$ from $p\, = \, (0,0,0)\,$ is roughly $\,0.8\A(D_{R})\,$ for large values of $R$. Note that for this simplified calculation the extrinsic disc consists of a finite number of flat half-discs, each one of which has radius $\,\sqrt{R^{2}- \rho_{0}^{2}\,}\,$, where $\rho_{0} \, \leq \, R\,$ denotes the orthogonal distance to $p$ from the plane containing the half-disc.

We have discussed this particular example at some length, because it seems to be a good example for displaying
in purely geometric terms what goes on at or close to
the otherwise still quite unknown borderline between hyperbolic and parabolic surfaces in $\mathbb{R}^{3}$.
In other words, the tangency function $\mathcal{T}$ introduced here seems to have an interesting
and instrumental r\^{o}le to play
concerning the quest of finding a
necessary and sufficient condition for a surface to be hyperbolic, resp. parabolic.
\end{remark}

\begin{corollary} \label{corBparab}
Let $P^{m}$ denote an $m-$dimensional radially mean $C-$ convex submanifold in the space form $\mathbb{K}^{n}_{b}$ of constant curvature $b \, <\, 0$. Suppose
\begin{equation} \label{eqBCcondition}
0 \, \leq \, m\left(\sqrt{-b} \, - \,  C\right) \, \leq \, \sqrt{-b} \quad ,
\end{equation}
and suppose that $P^{m}$ admits a
radial tangency bounding function $g(r)$ which satisfies $g(r) \, \geq \, \tilde{g}(r)$  for all sufficiently large $r$, where now - using shorthand notation $\tau \, = \, r\sqrt{-b}\,\,$:
\begin{equation} \label{eqGtildeB}
\tilde{g}(r) \, = \,  \sqrt{\frac{m\,r\log(r)\left(\sqrt{-b}\cosh(\tau)-C\sinh(\tau)\right)}
{\sinh(\tau)\log(r)+\tau\cosh(\tau)\log(r)+\sinh(\tau)}} \quad .
\end{equation}
Then $P^{m}$ is parabolic.
\end{corollary}

\begin{remark}\label{remNegative}
This result should likewise be compared with the fact established in \cite{MP3}, that every minimal
 submanifold of any co-dimension in a negatively curved space form is transient.
 Such submanifolds
are {\em{not}} $\,C-$convex, of course, for any positive $C$ (note that (\ref{eqBCcondition}) implies $C\, > \, 0$).
We observe that the bounding function
$\tilde{g}(r) \to 0$ for $r \to 0$ precisely when $\, C \, = \, \sqrt{-b}\,$.
Therefore, in relation to  the discussion in the previous remark \ref{remZero},
what 'induces' parabolicity in negatively curved ambient spaces
is to a large extent the radial mean $C\,-$con\-vexity assumption - not just the tangency condition $\mathcal{T}(x) \, \geq \, \tilde{g}(r(x))\,$.
 \end{remark}

\begin{proof}[Proof of Corollary \ref{corAparab}]
With $b \, = \, 0$ we have $Q_{b}(r) \, = \, r$ and since $\tilde{g}(r)$ is designed to satisfy
\begin{equation}
\frac{2}{r\,\tilde{g}^{2}(r)}\, = \, \frac{2r\log(r)+r}{r^{2}\log(r)} \, = \, \frac{d}{dr}\log(r^{2}\log(r))
\end{equation}
for sufficiently large values of $r$, say $r \, \geq \, A$,  we get
for some positive constant $c_{1}$:
\begin{equation}
\begin{aligned}
- \int_{A}^{r} \frac{2}{t\,\tilde{g}^{2}(t)}\,dt \,
= \, - \log(r^{2}\log(r)) + c_{1} \quad ,
\end{aligned}
\end{equation}
so that, for some other positive constant $c_{2}$:
\begin{equation}
\begin{aligned}
&\int^{\infty} Q_{b}(r)\exp\left(-\,\int_{A}^{r}\frac{m\,\,\eta_{Q_{b}}(t)}{g^{2}(t)}\, dt\right)\,dr\, \\
&\geq \, \int^{\infty} r\,c_{2}\exp\left(-\,\int_{A}^{r}\frac{2}{t\,\tilde{g}^{2}(t)}\, dt\right)\,dr\, \\
&= \, \int^{\infty} \frac{c_{2}}{r\log(r)}\,dr\, \\
&= \, \infty \quad ,
\end{aligned}
\end{equation}
which then implies parabolicity according to theorem \ref{thmAparab}.
\end{proof}

The proof of Corollary \ref{corBparab} follows essentially verbatim, except for handling the
allowed $C-$interval for given $b$ and $m$. The condition (\ref{eqBCcondition})
simply stems from the two obvious conditions, that the square root defining
$\tilde{g}(r)$ in (\ref{eqGtildeB}) must be well-defined and less than $1$. It is of independent
interest to note as well, that when $b$ approaches $0$ then $C$ must go to $0$, i.e. we are then back in the case of
Corollary \ref{corAparab}.\\

%%%%%%%%%%%%%%%%%%%%%%%%%%%%%%%%%%%%%%%%%%%%%%%%%%%%%%%%%%%%%%%%%%%%%%%%%%
%%%%%%%%%%%%%%%%%%%%%%%%%%%%%%%%%%%%%%%%%%%%%%%%%%%%%%%%%%%%%%%%%%%%%%%%%%
%%%%%%%%%%%%%%%%%%%%%%%%%%%%%%%%%%%%%%%%%%%%%%%%%%%%%%%%%%%%%%%%%%%%%%%%%%
%%%%%%%%%%%%%%%%%%%%%%%%%%%%%%%%%%%%%%%%%%%%%%%%%%%%%%%%%%%%%%%%%%%%%%%%%%

\section{Previous Results for Surfaces} \label{secPrevRes}

We show in the next section, that a number of surfaces, including the catenoid
and the hyperboloid of one sheet, are parabolic using the condition established
in Corollary \ref{corAparab}.
Parabolicity of those surfaces is known already from
criteria due to by Milnor and Ichihara, which we
briefly outline here for comparison with our Corollary \ref{corAparab}
- again mainly in order to support
and facilitate intuition before stating and proving the full versions
of our main results.

J. Milnor considered 2D warped products (see Definition \ref{defModel} in section \ref{secGeoBounds}) with general metrics
molded by a warping functions $w(r)\,$ on $r \in [\,0, \infty\,[\,$ as follows:
\begin{equation} \label{eq2Dwarp}
ds^{2}\,=\,dr^{2} + w^{2}(r)\,d\theta \quad .
\end{equation}
\begin{theorem}[Milnor, \cite{MilnorDecide}] \label{thmMilnorCrit}
If the Gaussian curvature function $K(s)$ considered as a function of the distance $s$ from the pole of
a model surface given as in (\ref{eq2Dwarp})
satisfies
\begin{equation} \label{eqMilnorCrit}
K(s) \, \geq \, -\frac{1}{s^{2}\log(s)} \quad ,
\end{equation}
then the surface is parabolic.
\end{theorem}

A condition for parabolicity in terms of the total curvature of a given surface
(not necessarily a model surface of revolution) is established by
K. Ichihara:

\begin{theorem}[Ichihara, \cite{Ich1}] \label{thmIchiharaCrit}
If a 2-dimensional manifold $M^{2}$ has finite total absolute Gaussian curvature, i.e.
\begin{equation}
\int_{M^{2}}\,|K|\, d\mu \, < \, \infty \, \quad ,
\end{equation}
then it is parabolic.
\end{theorem}

%%%%%%%%%%%%%%%%%%%%%%%%%%%%%%%%%%%%%%%%%%%%%%%%%%%%%%%%%%%%%%%%%%%%%%%%%%
%%%%%%%%%%%%%%%%%%%%%%%%%%%%%%%%%%%%%%%%%%%%%%%%%%%%%%%%%%%%%%%%%%%%%%%%%%
%%%%%%%%%%%%%%%%%%%%%%%%%%%%%%%%%%%%%%%%%%%%%%%%%%%%%%%%%%%%%%%%%%%%%%%%%%
%%%%%%%%%%%%%%%%%%%%%%%%%%%%%%%%%%%%%%%%%%%%%%%%%%%%%%%%%%%%%%%%%%%%%%%%%%

\section{Examples and Benchmarking \\Surfaces of Revolution} \label{secExampSurfOfRev}
The examples we have in mind are classical but suitably modified to provide well defined
(and simple) extrinsic balls (discs) to exemplify our analysis.
We consider piecewise smooth radially mean $0-$convex surfaces
of revolution in $\mathbb{R}^{3}$ constructed as
follows. In the $(x, z)-$plane we consider the profile generating
curve consisting of a (possibly empty) line segment along the $x-$axis:
from $(0, 0)$ to $(a, 0)$ for some $a \in \, [\,0, \infty\,[$
together with a smooth curve $\Gamma(u) \, = \, (x(u), z(u))$, $u \in \, [\,0, c\,]$
for some $c \in \, ]\,0, \infty\,[\,$ with $x(u) \, > \, 0$ for all $u$ and $x(0) \, = \, a\,$.
The corresponding surfaces of revolution then (possibly) have a flat (bottom) disc of radius $a$,
see Figure 2.\\

\begin{figure}[t]
\includegraphics[width=20mm]{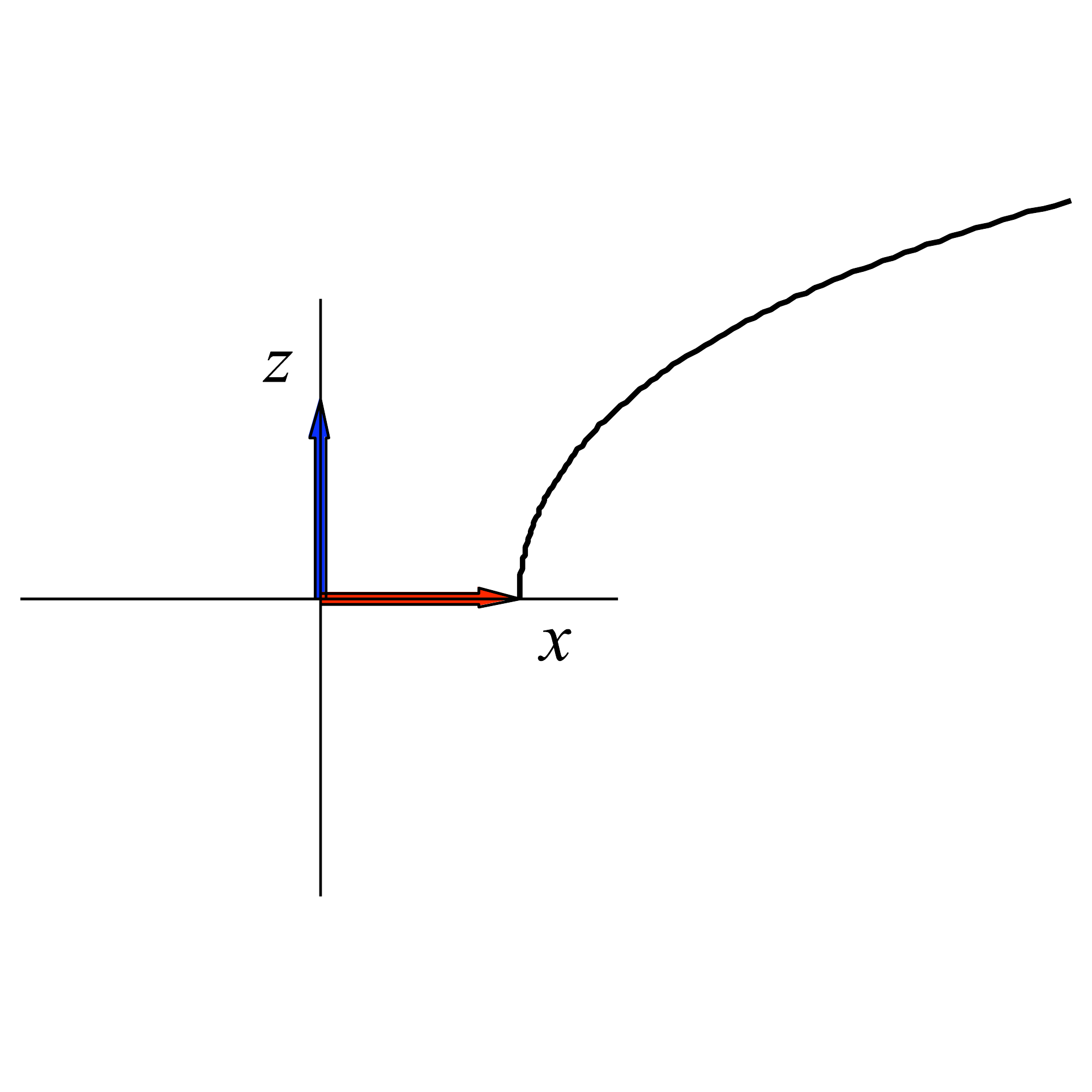} \quad
\includegraphics[width=20mm]{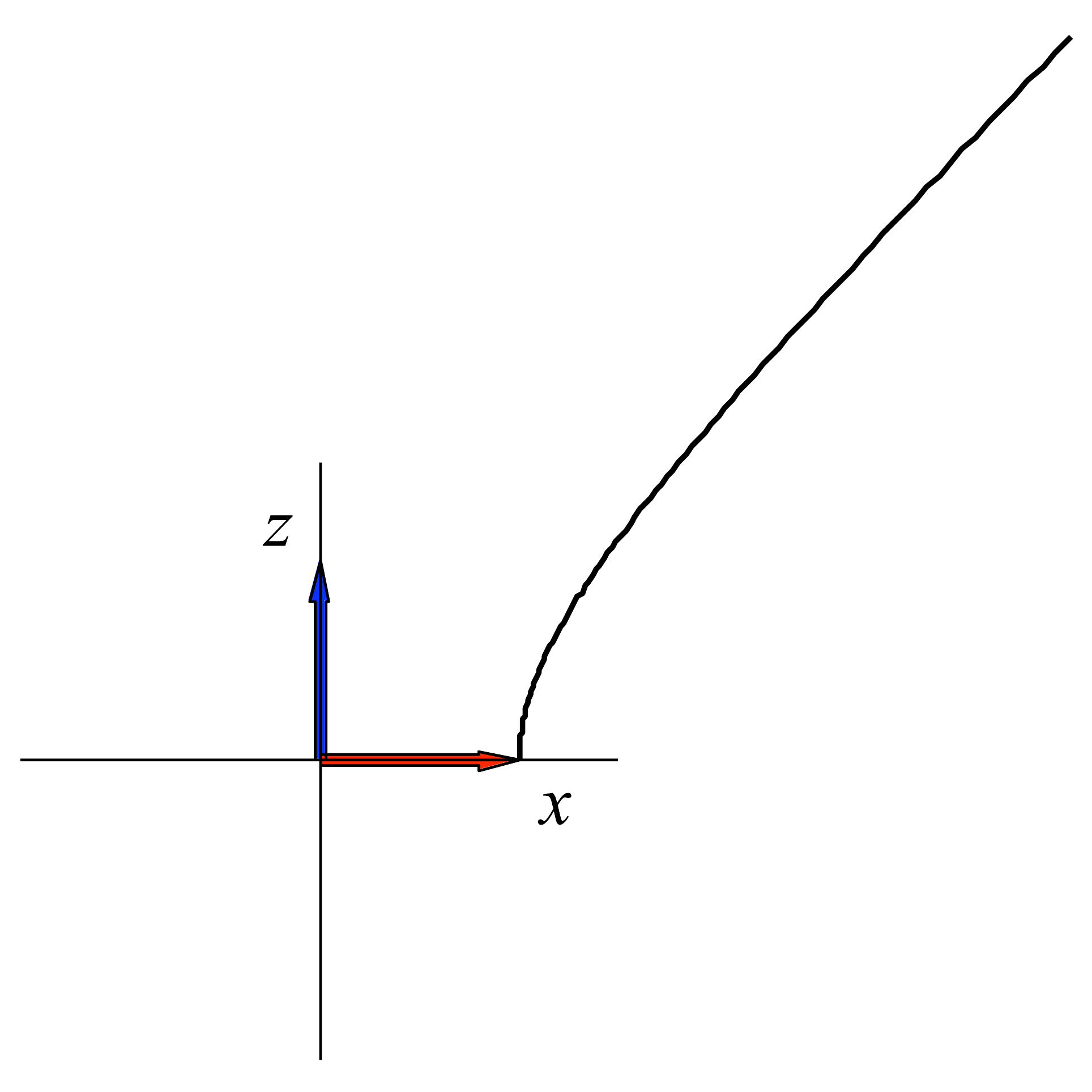} \quad
\includegraphics[width=20mm]{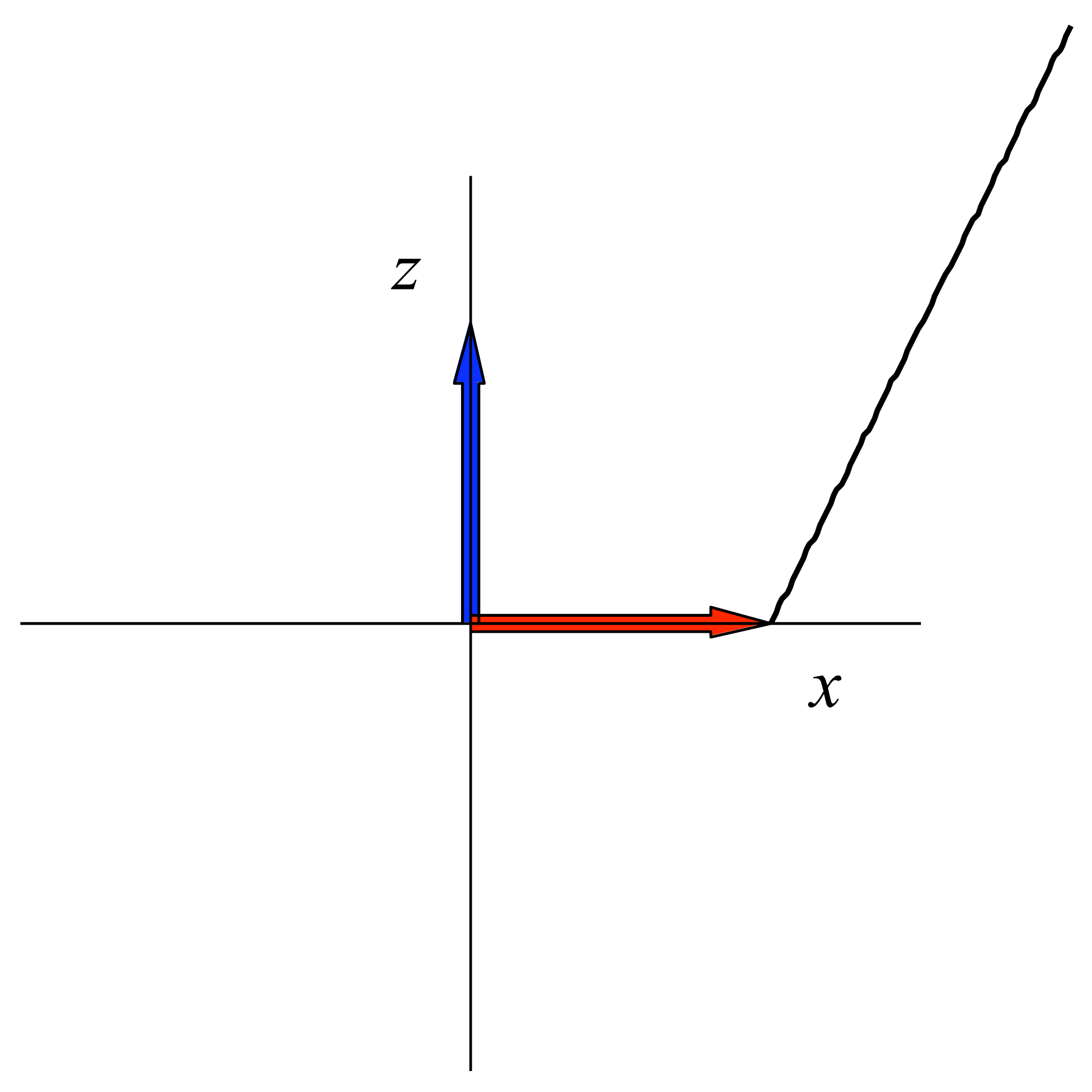} \quad
\includegraphics[width=20mm]{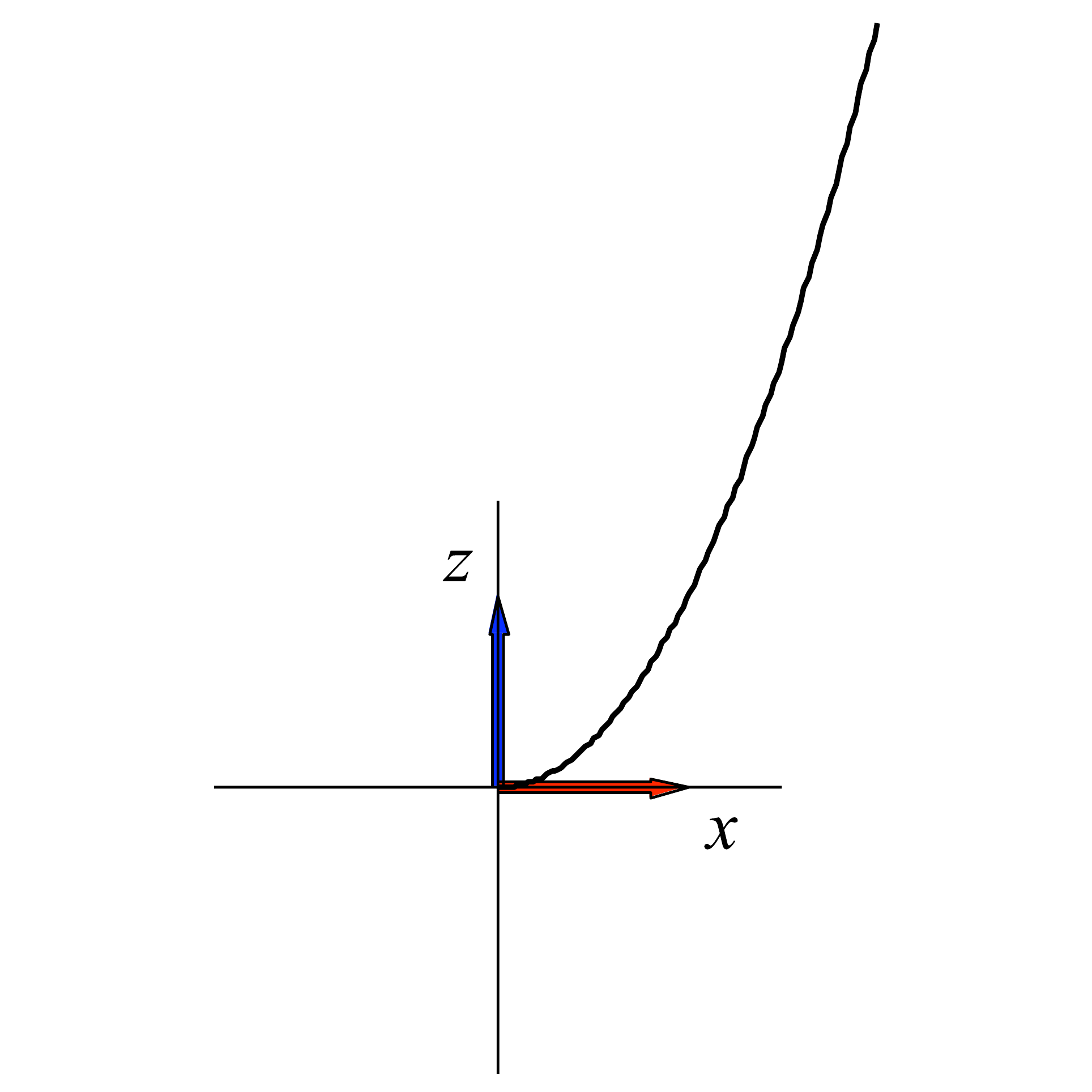} \quad
\includegraphics[width=20mm]{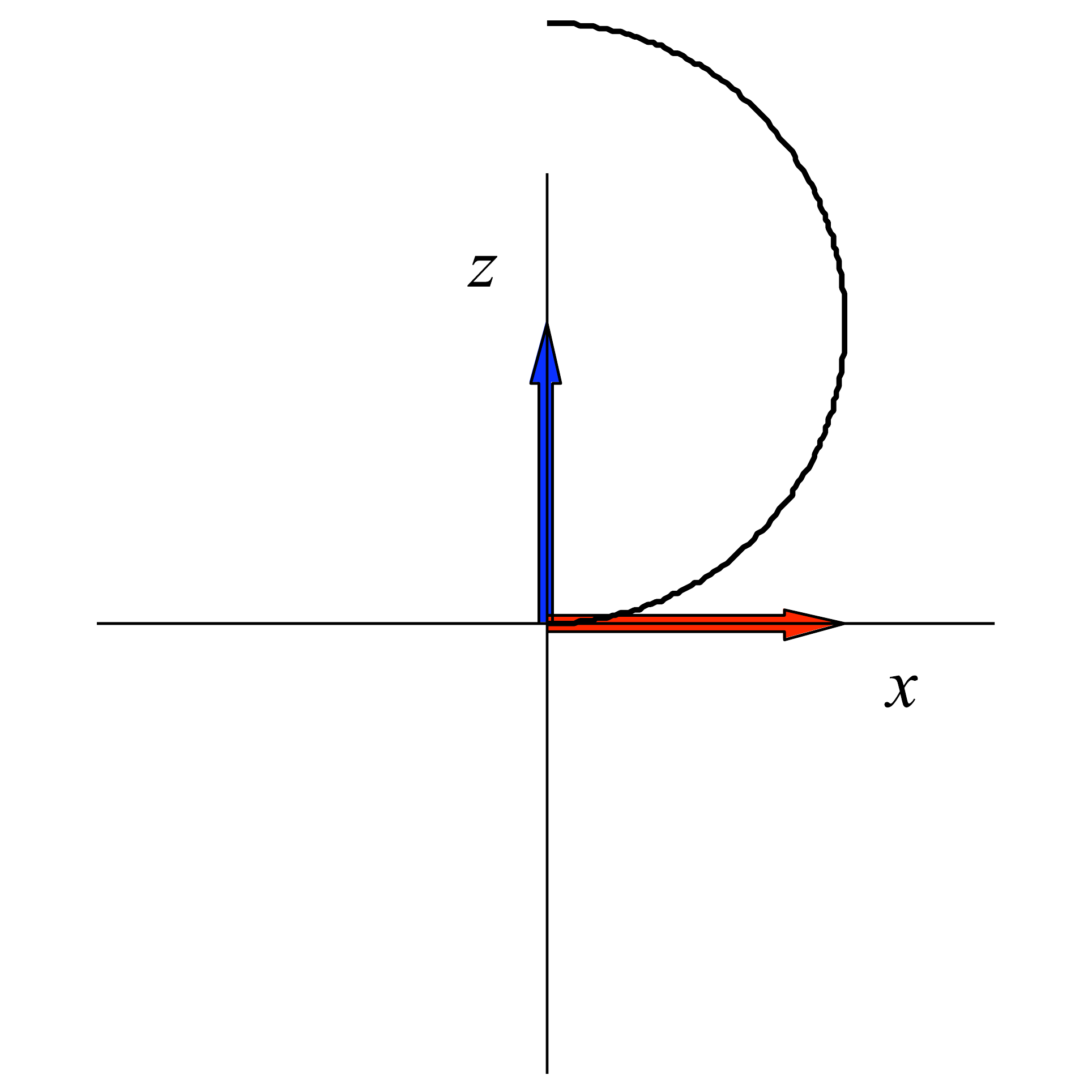} \quad \\
\includegraphics[width=20mm]{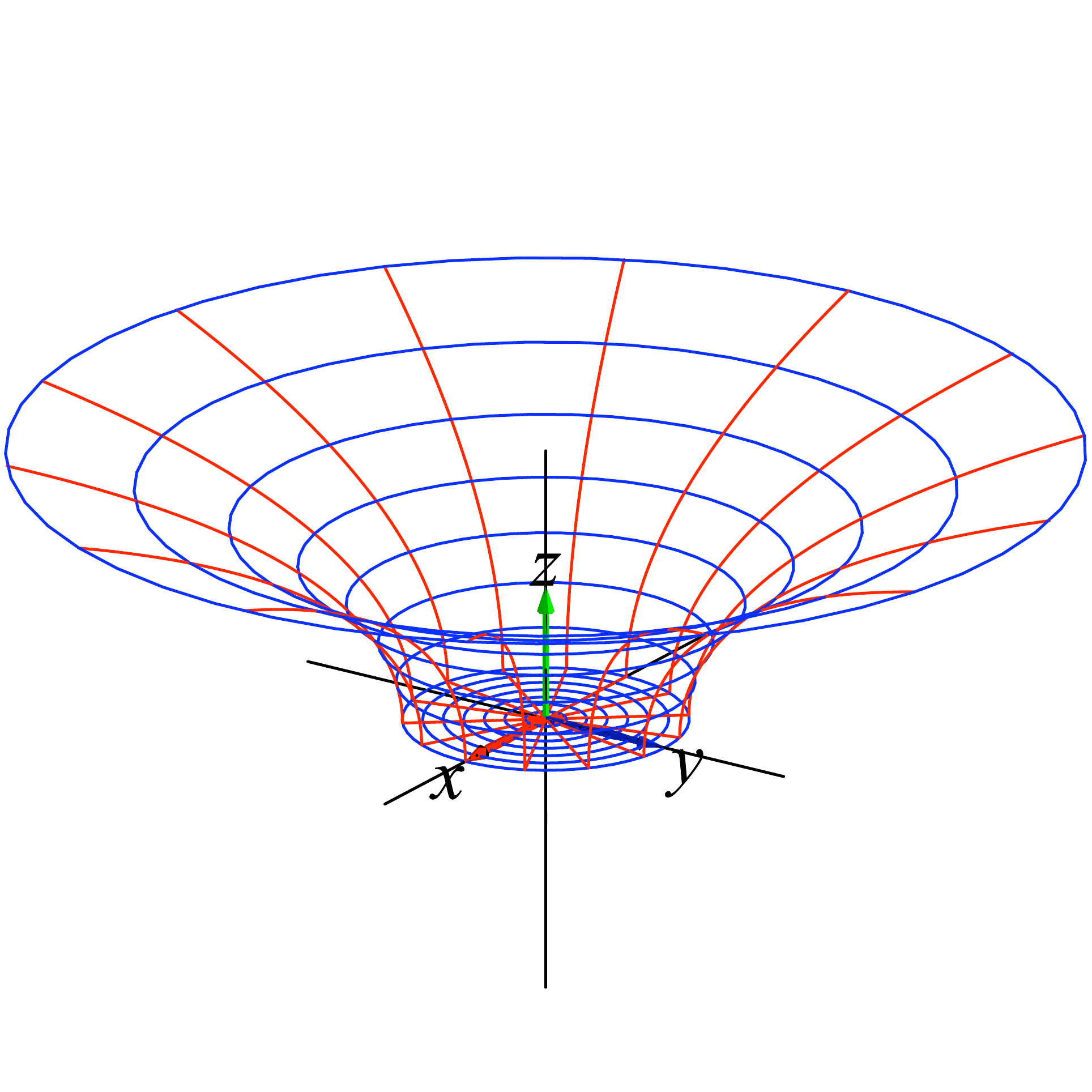} \quad
\includegraphics[width=20mm]{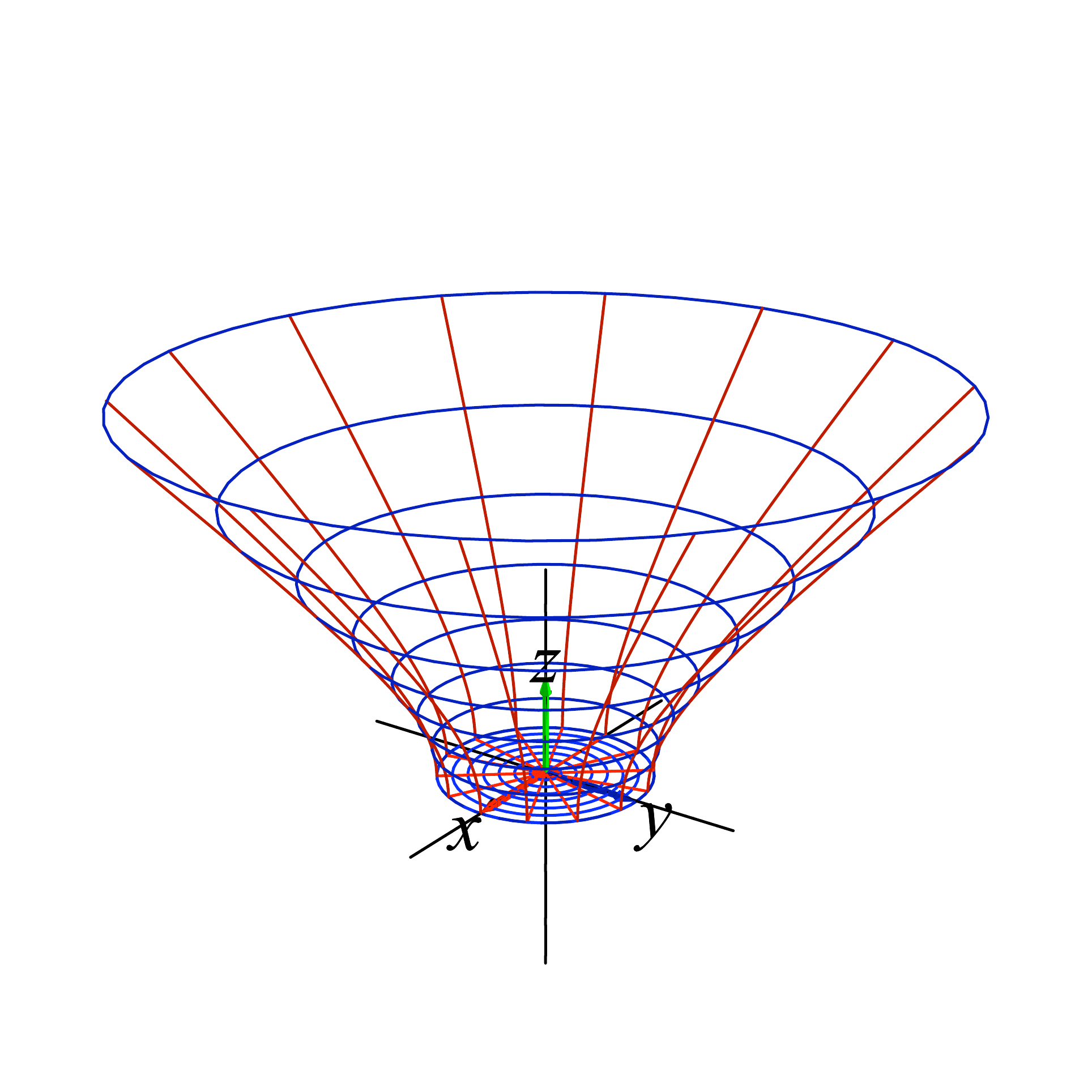} \quad
\includegraphics[width=20mm]{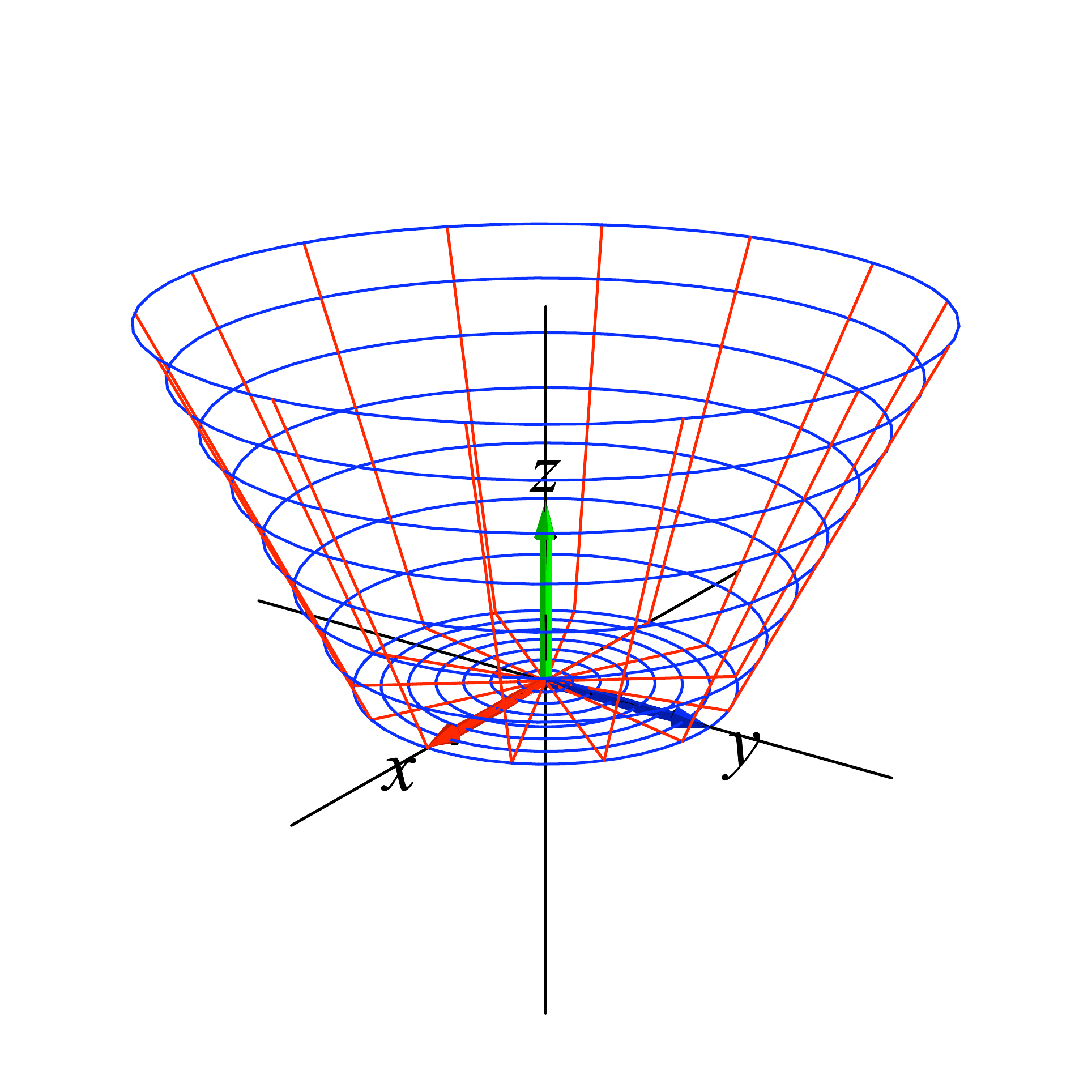} \quad
\includegraphics[width=20mm]{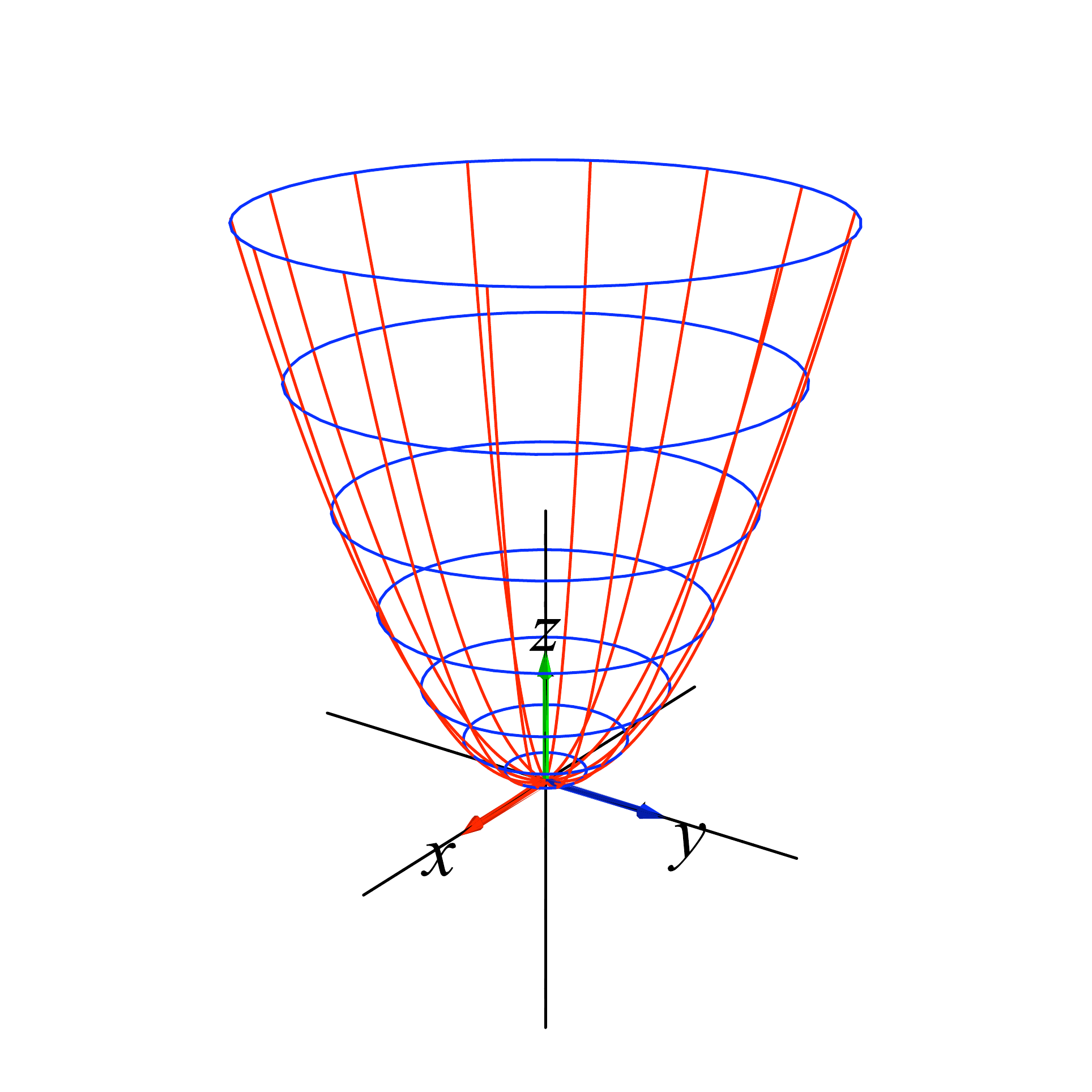} \quad
\includegraphics[width=20mm]{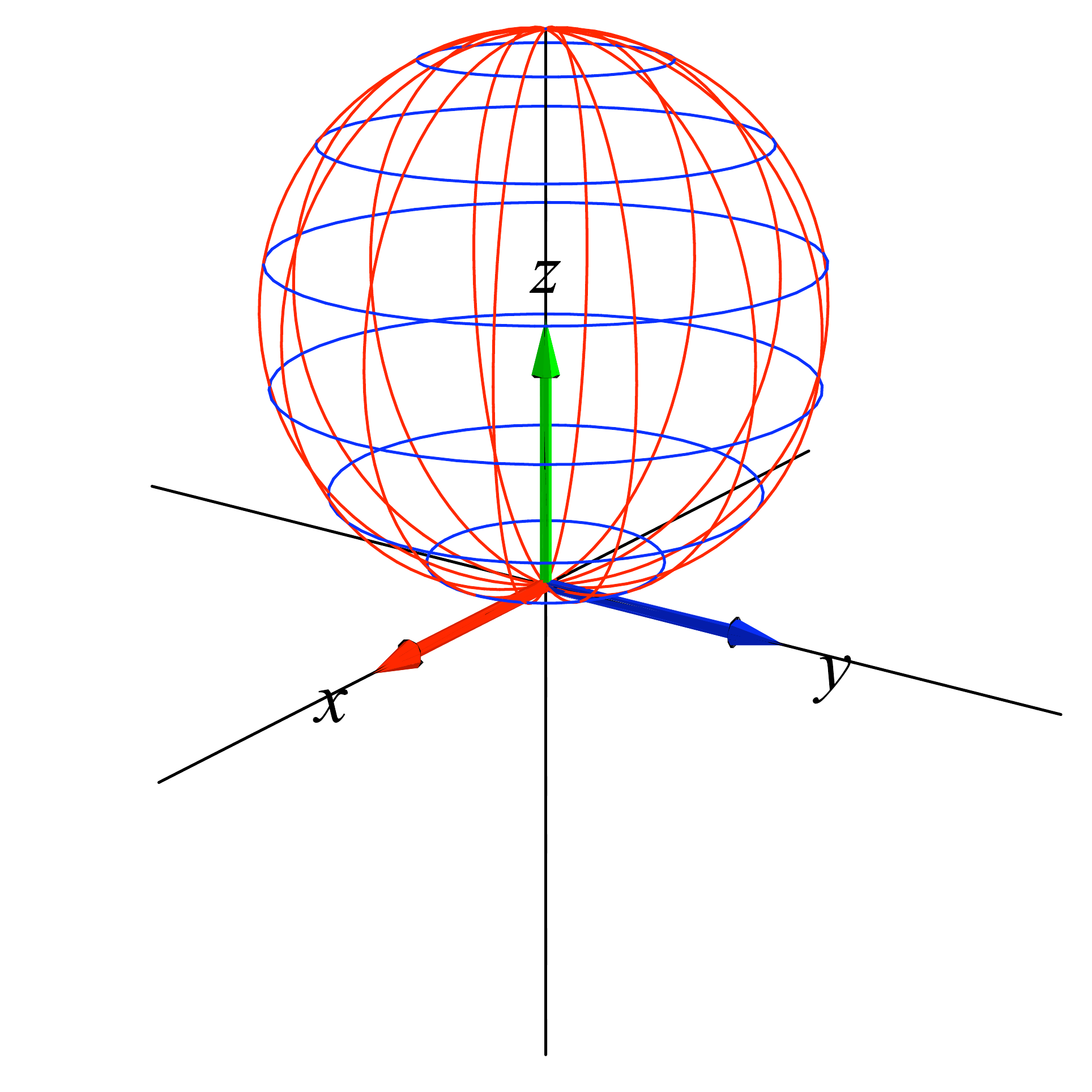}
\begin{center} \caption{Generating profile curves and the cor\-res\-pon\-ding surfaces of revolution.
Counting from the left the surfaces are constructed from a {\em{Catenoid}},
{\em{Hyperboloid of one sheet}}, {\em{Cone}}, {\em{Paraboloid}}, and a {\em{Sphere}}, respectively.
The three left-most surfaces
are completed with a flat bottom disc as shown. These surfaces are all radially mean $0-$convex from the Origin.
We illustrate in section \ref{secExampSurfOfRev} how the quite different criteria of Milnor, Ichihara,
as well as the condition of Corollary \ref{corAparab} in the present paper show that
all these surfaces are parabolic.}
\end{center} \label{figSurfsOfRev}
\end{figure}

The center of this disc, i.e. the origin $p\, = \, (0,0,0)$, will serve as the point from which the surfaces
under consideration will be $p-$radially mean $0-$convex as well as radially symmetric via the
specific choices of generating functions $x(u)$ and $z(u)$.\\

The extrinsic radius of the defining cutting sphere centered at $p$ is chosen to be
$R(c) \, > \, a$, so that the corresponding extrinsic disc containing $p$ consists of
the flat bottom disc of radius $a$ together with the following non-vanishing part of the surface of revolution:
\begin{equation*}
\Omega_{c}\,:\,\, r(u,v) \, = \, (x(u)\cos(v), x(u)\sin(v), z(u))\, , \,\,
u \in [\,0, c\,] \, , \,\, v \in [-\pi, \pi[ \, .
\end{equation*}
The area of this extrinsic disc is then a function of $a$ and $c$ as follows:
\begin{equation*}
\A(D_{R(c)})\, = \, A(a,c)\, = \, \pi\,a^2 + 2\,\pi\,\int_{0}^{c}x(t)\sqrt{x'^{\,2}(t) + z'^{\,2}(t)} \,\, dt \quad ,
\end{equation*}
and the length of its
boundary is simply
\begin{equation*}
 \Len(\partial D_{R(c)}) \, = \, L(a,c) \, = 2\pi\,x(c) \quad.
\end{equation*}
The exact isoperimetric quotient of the extrinsic disc of the resulting surface of revolution is thus
\begin{equation}
\mathcal{Q}\, = \, \mathcal{Q}(a,c) \, = \,
 \frac{x(c)}{(a^{2}/2) + \int_{0}^{c}x(t)\sqrt{x'^{\,2}(t) + z'^{\,2}(t)} \,\, dt} \quad .
\end{equation}

The main theorem of the present paper asserts that the quotient $\mathcal{Q}(a,c)$ is bounded from above by
the  right hand side of equation (\ref{eqIsopGeneralFirst2D}) - under the assumed condition, that the surface of revolution
is radially mean $0-$convex from $p\,$:
\begin{equation} \label{eqLambdaIneq}
\frac{2}{R(c)\,g(R(c))} \,
= \,\frac{2}{g(c)\,\sqrt{x^{2}(c)+z^2(c)}} \quad,
\end{equation}
where we use $g(u_{0})$ as shorthand for $g(R(u_{0}))$ which is the exact common radial tangency value for the surface
of revolution at each point of the circle $(x(u_{0})\cos(v), x(u_{0})\sin(v), z(u_{0}))$.
It is by definition the inner product:
\begin{equation}
\begin{aligned} \label{eqTan}
g(u_{0}) \, &= \, \langle \,\frac{\Gamma'(u_{0})}{\Vert \Gamma'(u_{0}) \Vert}\, , \,\,
\frac{\Gamma(u_{0})}{\Vert \Gamma(u_{0}) \Vert} \, \rangle \\
&= \, \frac{x(u_{0})\,x'(u_{0}) +  z(u_{0})\,z'(u_{0})}{\sqrt{x'^{\,2}(u_{0})+ z'^{\,2}(u_{0})}\,\,\sqrt{x^{2}(u_{0})+z^{2}(u_{0})}} \quad .
\end{aligned}
\end{equation}
Therefore the comparison upper bound reduces to:
\begin{equation} \label{eqUpperBdComp}
\frac{2}{R(c)\,g(R(c))} \,
= \,  \frac{2\,\sqrt{x'^{\,2}(c)+z'^{\,2}(c)}}{x(c)\,x'(c) +  z(c)\,z'(c)} \quad.
\end{equation}
The examples below then serve as illustrations of our main theorem within the category of
surfaces of revolution. Given the functions $x(u)$ and $z(u)$ we simply verify the following
inequality in each case:
\begin{equation}
\mathcal{Q}(a,c)\, = \, \frac{A(a,c)}{L(a,c)} \, \leq \, \frac{2}{R(c)\,g(c)} \quad ,
\end{equation}
which (as shown above) is equivalent to:
\begin{equation} \label{eqXZ}
\frac{x(c)}{(a^{2}/2) + \int_{0}^{c}x(t)\sqrt{x'^{\,2}(t) + z'^{\,2}(t)} \,\, dt} \, \leq \,
\frac{2\,\sqrt{x'^{\,2}(c)+z'^{\,2}(c)}}{x(c)\,x'(c) +  z(c)\,z'(c)} \quad .
\end{equation}
Simultaneously in our short presentations below we also illustrate how the
conditions of Corollary \ref{corAparab}, Theorems \ref{thmMilnorCrit} and \ref{thmIchiharaCrit} apply as well to prove
parabolicity of each one of the surfaces in question.

%%%%%%%%%%%%%%%%%%%%%%%%%%%%%%%%%%%%%%%%%%%%%%%%
%  CATENOID
%%%%%%%%%%%%%%%%%%%%%%%%%%%%%%%%%%%%%%%%%%%%%%%%

\begin{example} \label{exampCatenoid}
The {\em{Catenoid\,}} is a minimal surface and hence radially mean $0-$convex from any point on the surface;
It is also clearly radially mean $0-$convex from $p\,$ in its
truncated version considered here - completed with a "flat disc bottom" of radius $a\, = \, 1$.  We have in this particular case:
\begin{equation}
\begin{aligned}
x(u) \, &= \, \cosh(u) \\
z(u) \, &= \, u
\end{aligned}
\end{equation}
for $u \in \, [\,0, c\,]$.
Our isoperimetric inequality (in the form of (\ref{eqXZ})) is easily verified for all $c$.
The tangency function follows from equation (\ref{eqTan}) and is given by:
\begin{equation}
\mathcal{T}(u, v) \, = \, g(u) \, = \, \frac{\sinh(u)\cosh(u)+u}{\cosh(u)\sqrt{\cosh(u)^2+u^2}} \quad ,
\end{equation}
and the corresponding distance function to $p$, the origin, is:
\begin{equation}
r(u, v)\, = \, \sqrt{\cosh(u)^{2}+u^{2}} \quad ,
\end{equation}
It is a simple calculation to see that indeed the lower tangency bound of Corollary \ref{corAparab}, (\ref{eqCorA}),
holds true, so that we can conclude parabolicity of the truncated catenoid.
Moreover,
the Gaussian curvature of the catenoid is:
\begin{equation}
K(u,v) \, = \, -\frac{1}{\cosh^{4}(u)}
\end{equation}
The surface integral of the absolute curvature is clearly finite. A direct calculation gives:
\begin{equation}
\int_{M^{2}} \vert K \vert \,d\mu \, = \, 2\pi\int_{-\infty}^{\,\infty}\frac{1+\cosh(2u)}{2\cosh^{4}(u)}\, du \, = \, 4\pi \quad .
\end{equation}
Icihara's condition in Theorem \ref{thmIchiharaCrit} then also gives parabolicity.
To apply Milnor's condition, we calculate the arc length presentation:
\begin{equation}
s(u) \, = \, \sinh(u) \quad ,
\end{equation}
and parabolicity then follows from (\ref{eqMilnorCrit}), which in this case reads:
\begin{equation}
-\frac{1}{\cosh^{4}(u)} \, \geq \, -\,\frac{1}{\sinh^{2}(u)\log(\sinh(u))}  \quad .
\end{equation}

\end{example}

%%%%%%%%%%%%%%%%%%%%%%%%%%%%%%%%%%%%%%%%%%%%%%%%
%  HYPERBOLOID OF ONE SHEET
%%%%%%%%%%%%%%%%%%%%%%%%%%%%%%%%%%%%%%%%%%%%%%%%

\begin{example} \label{exampHyperboloid}
The truncated and completed {\em{Hyperboloid of one sheet\,}} is also radially
mean $0-$convex from the center point $p$ of its flat bottom, although this is not
clear from its shape. Indeed, from the generating functions
\begin{equation}
\begin{aligned}
x(u) \, &= \, \sqrt{1+u^2} \\
z(u) \, &= \, u
\end{aligned}
\end{equation}
 a short
calculation reveals the following non-negative radial mean convexity function - see also remark \ref{remRadCon} :
\begin{equation}
\mathcal{C}(x) \, = \, h(u(x)) \,= \, \frac{u^{2}}{(2u^{2}+1)^{5/2}} \, \geq \, 0 \quad .
\end{equation}

The inequality (\ref{eqXZ}) holds true for all $b \, \geq \, 0$ and our isoperimetric inequality is
thus verified in this case.
The curvature is
\begin{equation}
K(u,v) \, = \, -\,\frac{1}{(1+2u^{2})^{2}}
\end{equation}
and the intrinsic distance is essentially:
\begin{equation}
s(u) \, = \, i\,E(-u\,i, \sqrt{2})\quad ,
\end{equation}
where $E(z, k)$ is the incomplete elliptic integral of the second kind.
Parabolicity of the standard hyperboloid of one sheet then follows from
Milnor's condition:
\begin{equation}
-\,\frac{1}{(1+2u^{2})^{2}} \, \geq \,   -\frac{1}{s^{2}(u)\log(s(u))}   \quad.
\end{equation}
The surface integral of the absolute value of the Gauss curvature is again finite.
A direct calculation gives:
\begin{equation}
\int_{M^{2}} \vert K \vert \,d\mu \, = \, 2\pi\int_{-\infty}^{\,\infty}\frac{1}{(1+2\,u^{2})^{3/2}}\, du \, = \, 2\pi\sqrt{2} \quad .
\end{equation}
The tangency function is
\begin{equation}
\mathcal{T}(u, v) \, = \, \frac{2u\sqrt{1+u^2}}{1+2u^{2}} \quad ,
\end{equation}
and the corresponding extrinsic distance function to the origin is:
\begin{equation}
r(u, v)\, = \, \sqrt{1+2u^{2}} \quad .
\end{equation}
It is a simple calculation to see that indeed the lower tangency bound of (\ref{eqCorA})
holds true, so that parabolicity also stems from that condition.\\

\end{example}

%%%%%%%%%%%%%%%%%%%%%%%%%%%%%%%%%%%%%%%%%%%%%%%%%%%%%%%%%%%%%%%%%%%%%%%%%%
%  REMARK
%%%%%%%%%%%%%%%%%%%%%%%%%%%%%%%%%%%%%%%%%%%%%%%%%%%%%%%%%%%%%%%%%%%%%%%%%%

\begin{remark} \label{remRadCon}
Before considering the convex surfaces of revolution below - the paraboloids and spheres - we
note in passing, that for surfaces of revolution of the type considered above, i.e.
those with $z(u) \, = \, u$, we have the following differential inequality condition on $x(u)$, which is equivalent to
radial mean $0-$convexity from $p\,$ for the corresponding surface of revolution, $\mathcal{C}(u) \, \geq \, 0$:
\begin{equation} \label{eqConditRadCon}
\left(x(u)\,x''(u) - x'^{\,2}(u) -1\right)\left(x(u)- u\,x'(u)\right) \, \geq \, 0 \quad .
\end{equation}
In particular, for the hyperboloid, $x(u) \, = \, \sqrt{1+u^2}$, the left hand side of this
inequality reads $\, 2\,u^2/(1+u^2)^{3/2}\, > \, 0$. Moreover, we note that equality is obtained in the
condition (\ref{eqConditRadCon}) (assuming $x(0)\,=\,1$ and $x'(0)\,=\,0$) precisely when
$x(u)\, = \, \cosh(u)$, i.e. for the catenoid.\\
\end{remark}

%%%%%%%%%%%%%%%%%%%%%%%%%%%%%%%%%%%%%%%%%%%%%%%%
%  CONES
%%%%%%%%%%%%%%%%%%%%%%%%%%%%%%%%%%%%%%%%%%%%%%%%

\begin{example} \label{exampCone}
The truncated and completed {\em{Cone}} in Figure 2 is  radially mean $0-$convex from the point $p$.  We have:
\begin{equation}
\begin{aligned}
x(u) \, &= \, 1+ u\cos(\theta) \\
z(u) \, &= \, u\sin(\theta)
\end{aligned}
\end{equation}
for $u \in \, [\,0, c\,]$ and any constant $\theta \in \, [\,-\pi/2\, , \, \pi/2\,]$.
The inequality (\ref{eqXZ}) holds true for all $c \, \geq \, 0$ and all $\theta$,
thus verifying our main isoperimetric
inequality.
The tangency function is:
\begin{equation}
\mathcal{T}(u, v) \, = \,  \frac{\cos(\theta)+u}{\sqrt{1+2u\cos(\theta)+u^2}} \quad ,
\end{equation}
and the corresponding distance function to the origin is:
\begin{equation}
r(u, v)\, = \, \sqrt{1+2u\cos(\theta)+u^{2}} \quad .
\end{equation}
Again it is a simple calculation to see that the lower tangency bound of (\ref{eqCorA})
holds true:
\begin{equation}
\frac{\cos(\theta)+u}{\sqrt{1+2u\cos(\theta)+u^2}}\, \geq \,
\sqrt{\frac{\log(1+2u\cos(\theta)+u^{2})}{\log(1+2u\cos(\theta)+u^{2})+1}} \quad ,
\end{equation}
so that parabolicity follows from our Corollary \ref{corAparab}. Milnor's criterion is clearly satisfied as well
since the Gaussian curvature is non-negative.
Also, the (truncated) cones clearly  have finite total absolute curvature, so Ichihara's criterion applies immediately as well.
\end{example}

%%%%%%%%%%%%%%%%%%%%%%%%%%%%%%%%%%%%%%%%%%%%%%%%
%  PARABOLOIDS
%%%%%%%%%%%%%%%%%%%%%%%%%%%%%%%%%%%%%%%%%%%%%%%%

\begin{example} \label{exampParab}
The {\em{Paraboloid}} is clearly convex and hence radially mean $0-$convex from the top point $p\,$:
\begin{equation}
\begin{aligned}
x(u) \, &= \, u \\
z(u) \, &= \, \alpha\,u^{2}
\end{aligned}
\end{equation}
for $u \in \, [\,0, c\,]$ and any constant $\alpha \in \mathbb{R}$.
The inequality (\ref{eqXZ}) holds true for all $c \, \geq \, 0$ and all $\alpha$.
The tangency function is:
\begin{equation}
\mathcal{T}(u, v) \, = \, \frac{1+2\alpha^{2}u^{2}}{\sqrt{4\alpha^{2}u^{2}+1}\,\,\sqrt{1+\alpha^{2}u^{2}}} \quad ,
\end{equation}
and the since corresponding distance function to the origin is
\begin{equation}
r(u, v)\, = \, u\,\sqrt{1+\alpha^{2}u^{2}} \quad ,
\end{equation}
it is a simple calculation to see that indeed the lower tangency bound of (\ref{eqCorA})
holds true, so that parabolicity follows.
\end{example}

%%%%%%%%%%%%%%%%%%%%%%%%%%%%%%%%%%%%%%%%%%%%%%%%
%  SPHERE
%%%%%%%%%%%%%%%%%%%%%%%%%%%%%%%%%%%%%%%%%%%%%%%%

\begin{example} \label{exampSphere}
The round {\em{Sphere}} is clearly  mean $0-$convex from the south pole $p$; we have:
\begin{equation}
\begin{aligned}
x(u) \, &= \, -\sin(u-\pi) \\
z(u) \, &= \, 1+ \cos(u-\pi)
\end{aligned}
\end{equation}
for $u \in \, [\,0, c\,]$, $c \, \leq \, \pi\,$.
In this interval the inequality (\ref{eqXZ}) reduces to:
\begin{equation}
\frac{\sin(c)}{1-\cos(c)} \, \leq \, \frac{2}{\sin(c)} \quad ,
\end{equation}
which is clearly satisfied for all $c \in \, [\,0, \pi\,]$, so that we have again thereby verified
our main isoperimetric inequality (\ref{eqIsopGeneralFirst}) in the reduced form of (\ref{eqIsopGeneralFirst2D}).
Parabolicity follows in this (extreme) case from mere compactness of the complete surface of a sphere.\\

\end{example}

%%%%%%%%%%%%%%%%%%%%%%%%%%%%%%%%%%%%%%%%%%%%%%%%%%%%%%%%%%%%%%%%%%%%%%%%%%
%%%%%%%%%%%%%%%%%%%%%%%%%%%%%%%%%%%%%%%%%%%%%%%%%%%%%%%%%%%%%%%%%%%%%%%%%%
%%%%%%%%%%%%%%%%%%%%%%%%%%%%%%%%%%%%%%%%%%%%%%%%%%%%%%%%%%%%%%%%%%%%%%%%%%
%%%%%%%%%%%%%%%%%%%%%%%%%%%%%%%%%%%%%%%%%%%%%%%%%%%%%%%%%%%%%%%%%%%%%%%%%%
%%%%%%%%%%%%%%%%%%%%%%%%%%%%%%%%%%%%%%%%%%%%%%%%%%%%%%%%%%%%%%%%%%%%%%%%%%
%%%%%%%%%%%%%%%%%%%%%%%%%%%%%%%%%%%%%%%%%%%%%%%%%%%%%%%%%%%%%%%%%%%%%%%%%%

\section{Geometric Bounds from Below \\ and Model Space Carriers} \label{secGeoBounds}

%%%%%%%%%%%%%%%%%%%%%%%%%%%%%%%%%%%%%%%%%%%%%%%%%%%%%%%%%%%%%%%%%%%%%%%%%%
%%%%%%%%%%%%%%%%%%%%%%%%%%%%%%%%%%%%%%%%%%%%%%%%%%%%%%%%%%%%%%%%%%%%%%%%%%

\subsection{Lower Mean Convexity Bounds } \label{secConvTang}
Given an immersed, complete $m-$di\-men\-sio\-nal submanifold
$P^m$ in a complete Riemannian manifold $N^n$
with a pole $p$, we denote the distance function
from $p$ in the ambient space $N^{n}$ by $r(x) =
\dist_{N}(p, x)$ for all $x \in N$. Since $p$ is
a pole there is - by definition - a unique
geodesic from $x$ to $p$ which realizes the
distance $r(x)$. We also denote by $r$ the
restriction $r\vert_P: P\longrightarrow \erre_{+}
\cup \{0\}$. This restriction is then called the
extrinsic distance function from $p$ in $P^m$.
The corresponding extrinsic metric balls of
(sufficiently large) radius $R$ and center $p$
are denoted by $D_R(p) \subseteq P$ and defined
as follows:
$$D_{R}(p) = B_{R}(p) \cap P =\{x\in P \,|\, r(x)< R\} \quad ,$$
where $B_{R}(p)$ denotes the geodesic $R$-ball around the pole $p$ in $N^n$. The
extrinsic ball $D_R(p)$ is assumed throughout to be a connected,  pre-compact domain in $P^m$ which contains the
pole $p$ of the ambient space. Since $P^{m}$ is (unless the contrary is clearly stated) assumed to be unbounded in $N$ we
have for every sufficiently large $R$ that $B_{R}(p) \cap P \neq P$.

In order to control the mean curvatures $H_P(x)$ of $P^{m}$ at distance $r$ from
$p$ in $N^{n}$ we introduce the following definition:

\begin{definition} The $p$-radial mean curvature function for $P$ in $N$
is defined in terms of the inner product of $H_{P}$ with the $N$-gradient of the
distance function $r(x)$ as follows:
$$
\mathcal{C}(x) = -\langle \nabla^{N}r(x),
H_{P}(x) \rangle  \quad {\textrm{for all}}\quad x
\in P \,\, .
$$

We say that the submanifold $P$ satisfies a {\it radial mean convexity
 condition} from $p\in P$ when we have a smooth function
 $h: P \mapsto \erre\,\, ,$ such that
\begin{equation}
\mathcal{C}(x) \, \geq \, h(r(x))  \quad {\textrm{for all}}
\quad x \in P \,\, .
\end{equation}
The submanifold $P$ is called {\it radially mean $C-$convex} when
its $p-$radial mean curvature is bounded from below by the constant $C$.
Minimal submanifolds and convex hypersurfaces are radially mean $0-$convex, cf. the following remark.
\end{definition}

\begin{remark} \label{remHypCon}
A hypersurface $P^{n-1}$ in $\erre^{n}$ is said to be {\em{convex}} if, for every $q\in P$, the tangent
hyperplane $T_qP$  of $P$ at $q$ does not separate $P$ into two parts, see \cite[Vol. II, p. 40\,]{KobNom},
and \cite[Vol III, p. 93\,]{Sp}. Convexity extends to hypersurfaces in $\mathbb{K}^{n}_{b}$, $b\neq 0$, if
in the above definition we replace 'the tangent
hyperplane' by the totally geodesic
hypersurface $\exp_{q}(T_{q}P)$ in $\mathbb{K}^{n}_{b}$, see
\cite[Vol IV, pp.121-123\,]{Sp}.

A classical generalization of Hadamard's Theorem states: If $P$ is a convex hypersurface in
$\mathbb{K}^{n}_{b}$, then the Weingarten map of $P$, $L_{\xi}$, is semi-negative definite at  all points $q\in P$, see  \cite[vol IV, pp.121-123 and vol III, pp. 91-95\,]{Sp} and \cite{DCW}.
Here $\xi$ denotes the outward pointing unit normal of the hypersurface. Therefore
all the principal curvatures of $P$ will be non-positive according to the orientation
chosen for $P$. Given $x \in P$, the outward pointing unit normal $\xi(x)$, and the mean curvature normal vector $H_{P}(x)$ to $P$ at $x$ we then obtain radial mean $0-$convexity of $P$ from any given point $p \, \in P\,$:
\begin{equation}
\begin{aligned}
\mathcal{C}(x)\, &= \, -\langle \,H_{P}(x)\, ,\,\,\nabla^{\mathbb{K}^{n}_{b}} r(x)\,\rangle \\
&\, =\, \frac{1}{n-1}\sum_{i=1}^{n-1}\vert k_i(x)\vert\langle\,\xi(x),\nabla^{\mathbb{K}^{n}_{b}} r(x)\,\rangle \\
& \, \geq \, 0 \quad ,
\end{aligned}
\end{equation}
where $k_i(x)$ are the principal curvatures of $P$ at $x$ and we are using Proposition 2.3 in \cite{Pa1}.
\end{remark}

%%%%%%%%%%%%%%%%%%%%%%%%%%%%%%%%%%%%%%%%%%%%%%%%%%%%%%%%%%%%%%%%%%%%%%%%%%
%%%%%%%%%%%%%%%%%%%%%%%%%%%%%%%%%%%%%%%%%%%%%%%%%%%%%%%%%%%%%%%%%%%%%%%%%%
%%%%%%%%%%%%%%%%%%%%%%%%%%%%%%%%%%%%%%%%%%%%%%%%%%%%%%%%%%%%%%%%%%%%%%%%%%
%%%%%%%%%%%%%%%%%%%%%%%%%%%%%%%%%%%%%%%%%%%%%%%%%%%%%%%%%%%%%%%%%%%%%%%%%%

\subsection{Lower Tangency Bounds} \label{secLowerTangBound}

The final notion needed to describe our comparison setting is the idea of {\it radial tangency}.
If we denote by $\nabla^N r$ and $\nabla^P r$ the
 gradients of
$r$ in $N$ and $P$ respectively, then let us first remark that $\nabla^P r(q)$ is
just the tangential component in $P$ of $\nabla^N r(q)$, for all $q\in P$. Hence
we have the following basic relation:
\begin{equation}\label{eq2.1}
\nabla^N r =
\nabla^P r +(\nabla^P r)^\bot \quad ,
\end{equation}
where $(\nabla^P r)^\bot(q)$ is perpendicular to $T_qP$ for all $q\in P$.

Considering the extrinsic disc $D_{r(x)} \subset P$ and the outward pointing unit normal vector $\nu(x)$ to
the extrinsic sphere $\partial D_{r(x)}$, then
\begin{equation}
\nabla^{P}r(x) \, = \, \langle \nabla^{N}r(x) , \nu(x) \rangle\,\nu(x) \quad ,
\end{equation}
so that
$\Vert \nabla^{P}r(x)\Vert$ measures the local {\em{tangency}} to $P$ at $x$ of the geodesics
issuing from $p$. Full tangency means
$\Vert \nabla^{P}r(x)\Vert\, = \, \langle \nabla^{N}r(x) , \nu(x) \rangle \, = \, 1$,
i.e. $\nabla^{N}r(x)\, = \, \nu(x)$ and minimal tangency
means orthogonality, i.e. $\Vert \nabla^{P}r(x)\Vert\, = \, \langle \nabla^{N}r(x) , \nu(x) \rangle \, = \, 0\,$.\\

In order to control this tangency of geodesics to the submanifold $P$ we introduce the following

\begin{definition} \label{defRadTan}
We say that the submanifold $P$ satisfies a {\it radial tangency
 condition} from $p\in P$ when we have a smooth positive function
 $g: P \mapsto \erre_{+} \,\, ,$ such that
\begin{equation}
\mathcal{T}(x) \, = \, \Vert \nabla^P r(x)\Vert
\geq g(r(x)) \, > \, 0  \quad {\textrm{for all}}
\quad x \in P \,\, .
\end{equation}
\end{definition}

%%%%%%%%%%%%%%%%%%%%%%%%%%%%%%%%%%%%%%%%%%%%%%%%%%%%%%%%
%%%%%%%%%%%%%%%%%%%%%%%%%%%%%%%%%%%%%%%%%%%%%%%%%%%%%%%%
%%%%%%%%%%%%%%%%%%%%%%%%%%%%%%%%%%%%%%%%%%%%%%%%%%%%%%%%

\subsection{Auxiliary Model Spaces} \label{secModel}
The concept of a model space is of instrumental importance for the precise
statements of our comparison results. We therefore consider the
definition and some first well-known properties in some detail:

\begin{definition}[See \cite{Gri}, \cite{GreW}]
\label{defModel}  A $w-$model space $M_{w}^{m}$
is a smooth warped product with base $B^{1} =
[\,0,\, R[ \,\,\subset\, \mathbb{R}$ (where $\, 0
< R \leq \infty$\,), fiber $F^{m-1} =
S^{m-1}_{1}$ (i.e. the unit $(m-1)-$sphere with
standard metric), and warping function $w:\,
[\,0, \,R[\, \to \mathbb{R}_{+}\cup \{0\}\,$ with
$w(0) = 0$, $w'(0) = 1$, and $w(r) > 0\,$ for all
$\, r > 0\,$. The point $p_{w} = \pi^{-1}(0)$,
where $\pi$ denotes the projection onto $B^1$, is
called the {\em{center point}} of the model
space. If $R = \infty$, then $p_{w}$ is a pole of
$M_{w}^{m}$.
\end{definition}

\begin{remark}\label{propSpaceForm}
The simply connected space forms
$\mathbb{K}^{m}_{b}$ of constant curvature $b$ can
be constructed as  $w-$models with any given
point as center point using the warping functions
\begin{equation}
w(r) = Q_{b}(r) =\begin{cases} \frac{1}{\sqrt{b}}\sin(\sqrt{b}\, r) &\text{if $b>0$}\\
\phantom{\frac{1}{\sqrt{b}}} r &\text{if $b=0$}\\
\frac{1}{\sqrt{-b}}\sinh(\sqrt{-b}\,r) &\text{if
$b<0$} \quad .
\end{cases}
\end{equation}
Note that for $b > 0$ the function $Q_{b}(r)$
admits a smooth extension to  $r = \pi/\sqrt{b}$.
For $\, b \leq 0\,$ any center point is a pole.
\end{remark}

\begin{proposition}[See \cite{O'N} p. 206]\label{propWarpMean}
Let $M_{w}^{m}$ be a $w-$model with war\-ping
function $w(r)$ and center $p_{w}$. The distance
sphere of radius $r$ and center $p_{w}$ in
$M_{w}^{m}$, denoted as $S^w_r$, is the fiber
$\pi^{-1}(r)$. This distance sphere has the
following constant mean curvature vector in
$M_{w}^{m}$
\begin{equation}
H_{{\pi^{-1}}(r)} = -\eta_{w}(r)\,\nabla^{M}\pi =
-\eta_{w}(r)\,\nabla^{M}r \quad ,
\end{equation}
where the mean curvature function $\eta_{w}(r)$
is defined by
\begin{equation}\label{eqWarpMean}
\eta_{w}(r)  = \frac{w'(r)}{w(r)} =
\frac{d}{dr}\ln(w(r))\quad .
\end{equation}
\end{proposition}

In particular we have for the constant curvature
space forms $\mathbb{K}^{m}_{b}$:
\begin{equation}\label{eqHbversusWarp}
\eta_{Q_{b}}(r) = \begin{cases} \sqrt{b}\cot(\sqrt{b}\,r) &\text{if $b>0$}\\
\phantom{\sqrt{b}} 1/r &\text{if
$b=0$}\\\sqrt{-b}\coth(\sqrt{-b}\,r) &\text{if
$b<0$} \quad .
\end{cases}
\end{equation}

\begin{definition}
Let $p$ be a point in a Riemannian manifold $M$
and let $x \in M-\{ p \}$. The sectional
curvature $K_{M}(\sigma_{x})$ of the two-plane
$\sigma_{x} \in T_{x}M$ is then called a
\textit{$p$-radial sectional curvature} of $M$ at
$x$ if $\sigma_{x}$ contains the tangent vector
to a minimal geodesic from $p$ to $x$. We denote
these curvatures by $K_{p, M}(\sigma_{x})$.
\end{definition}

\begin{proposition}[See \cite{GreW} and \cite{Gri}] \label{propModelRadialCurv}
Let $M_{w}^{m}$ be a $w-$model with center point
$p_{w}$. Then the $p_{w}$-radial sectional
curvatures of $M_{w}^{m}$ at every $x \in
\pi^{-1}(r)$ (for $\,r\, >\, 0\,$) are all
identical and determined by the radial function
$\,K_{w}(r)\,$ defined as follows:
\begin{equation}
K_{p_{w} , \tN_{w}}(\sigma_{x}) \, = \, K_{w}(r)
\, = \, -\frac{w''(r)}{w(r)} \quad.
\end{equation}
\end{proposition}

%%%%%%%%%%%%%%%%%%%%%%%%%%%%%%%%%%%%%%%%%%%%%%%%%%%%%%%%%%%%%%%%%%%%%
%%%%%%%%%%%%%%%%%%%%%%%%%%%%%%%%%%%%%%%%%%%%%%%%%%%%%%%%%%%%%%%%%%%%%
%%%%%%%%%%%%%%%%%%%%%%%%%%%%%%%%%%%%%%%%%%%%%%%%%%%%%%%%%%%%%%%%%%%%%

For any given warping function $\,w(r)\,$ we
introduce the isoperimetric quotient function
$\,q_{w}(r)\,$  for the corresponding $w-$model
space $\,M_{w}^{m}\,$ as follows:
\begin{equation} \label{eqDefq}
q_{w}(r) \, = \,
\frac{\Vol(B_{r}^{w})}{\Vol(S_{r}^{w})} \, = \,
\frac{\int_{0}^{r}\,w^{m-1}(t)\,dt}{w^{m-1}(r)}
\quad ,
\end{equation}
where $\,B_{r}^{w}\,$ denotes the polar centered
geodesic $r-$ball of radius $\,r\,$ in
$\,M_{w}^{m}\,$ with boundary sphere
$\,S_{r}^{w}\,$.

When discussing isoperimetric comparison inequalities the warping function $\,w\,$ usually
appears raised to the power $\,m-1\,$. Accordingly we will use
the notion $\,a_{w}(r)\,$ to denote this power of
$\,w\,$ ($V_{0}$ denotes the volume of the unit sphere in $\mathbb{R}^{m}$):
\begin{equation}
\begin{aligned}
a_{w}(r) \, &= \, \Vol(S_{r}^{w})/V_{0} \, = \, w^{m-1}(r)\,\, , \quad \textrm{so that}\\
q_{w}(r) \, &=\,
\frac{\int_{0}^{r}\,a_{w}(t)\,dt}{a_{w}(r)} \quad
.
\end{aligned}
\end{equation}

%%%%%%%%%%%%%%%%%%%%%%%%%%%%%%%%%%%%%%%%%%%%%%%%%%%%%%%%%%%%%%%%%%%%%%%%%%%%%%%%
%%%%%%%%%%%%%%%%%%%%%%%%%%%%%%%%%%%%%%%%%%%%%%%%%%%%%%%%%%%%%%%%%%%%%%%%
%%%%%%%%%%%%%%%%%%%%%%%%%%%%%%%%%%%%%%%%%%%%%%%%%%%%%%%%%%%%%%%%%%%%%%%%%%
%%%%%%%%%%%%%%%%%%%%%%%%%%%%%%%%%%%%%%%%%%%%%%%%%%%%%%%%%%%%%%%%%%%%%%%%%%
%%%%%%%%%%%%%%%%%%%%%%%%%%%%%%%%%%%%%%%%%%%%%%%%%%%%%%%%%%%%%%%%%%%%%%%%%%
\subsection{Lower Curvature Bounds}
\label{secLowerCurv}

The 2.nd order analysis of the restricted distance function $r_{|_{P}}$ is firstly and foremost governed by the
following Hessian comparison theorem for manifolds with a pole:

\begin{theorem}[See \cite{GreW}, Theorem A]\label{thmGreW}
Let $N^{n}$ be a manifold with a pole $p$, let $M_{w}^{m}$ denote a $w-$model
with center $p_{w}$. Suppose that every $p$-radial sectional curvature at $x \in
N - \{p\}$ is bounded from above by the $p_{w}$-radial sectional curvatures in
$M_{w}^{m}$ as follows:
\begin{equation}\label{eqKbound}
\mathcal{K}(\sigma(x)) \, = \, K_{p,
N}(\sigma_{x}) \geq -\frac{w''(r)}{w(r)}
\end{equation}
for every radial two-plane $\sigma_{x} \in T_{x}N$ at distance $r = r(x) =
\dist_{N}(p, x)$ from $p$ in $N$. Then the Hessian of the distance function in
$N$ satisfies
\begin{equation}\label{eqHess}
\begin{aligned}
\Hess^{N}(r(x))(X, X) &\leq \Hess^{M_{w}^{m}}(r(y))(Y, Y)\\ &=
\eta_{w}(r)\left(1 - \langle \nabla^{M}r(y), Y \rangle_{M}^{2} \right) \\ &=
\eta_{w}(r)\left(1 - \langle \nabla^{N}r(x), X \rangle_{N}^{2} \right)
\end{aligned}
\end{equation}
for every unit vector $X$ in $T_{x}N$ and for every unit vector $Y$ in $T_{y}M$
with $\,r(y) = r(x) = r\,$ and $\, \langle \nabla^{M}r(y), Y \rangle_{M} =
\langle \nabla^{N}r(x), X \rangle_{N}\,$.
\end{theorem}

If we let $\mu\, : \, N \mapsto \mathbb{R}\, \, $ denote a smooth function on
$N$, then the restriction $\Tilde{\mu} \, = \, \mu_{|_{P}}$ is a smooth function
on $P$ and the respective Hessian operators, $\, \operatorname{Hess}^{N}(\mu) \,
$ and $\, \operatorname{Hess}^{P}(\Tilde{\mu}) \, $,   are related as follows:

\begin{proposition}[\cite{JK}]
\begin{equation} \label{eqHess1}
\begin{aligned}
 \operatorname{Hess}^{P}(\Tilde{\mu})(X, Y)\,
  = \,
    &\operatorname{Hess}^{N}(\mu)(X, Y) \, \\
    & + \, \langle \, \nabla^{N}(\mu), \, \alpha_{x}(X,
Y)\,  \rangle
\end{aligned}
\end{equation}
for all tangent vectors $\, X , \, Y  \in T_{x}P^{m} \subset T_{x}N^{n}\, $,
where $\, \alpha_{x} \,$ is the second fundamental form of $\, P \, $ at $x$ in
$\, N \,$.
\end{proposition}

If we modify $\mu$ by a smooth function $f: \mathbb{R} \to \mathbb{R}$ we then
get
\begin{corollary}[\cite{JK}] \label{corHessfomu}
\begin{equation}
 \begin{aligned}
  \operatorname{Hess}^{P}(f \circ \Tilde{\mu})(X, X) \,
  = \, \, \, &f''(\mu) \langle\nabla^{N}(\mu), X \, \rangle^{2} \\
  &+ \, f'(\mu) \operatorname{Hess}^{N}(\mu)(X, X) \\
  &+ \, \langle \nabla^{N}(\mu), \, \alpha_{x}(X, \, X) \,  \rangle    \quad .
 \end{aligned}
 \end{equation}
 for all $\, X \in T_{x}P^{m} \, $ .
\end{corollary}

\subsection{The Laplacian Comparison Space} \label{subsecLapCompSp}

Combining these results and tracing the resulting Hessian comparison statement
in an orthonormal basis of $T_{x}P^{m}$, we obtain the following Laplacian
inequality:

\begin{corollary} \label{corLapComp}
Suppose again that the assumptions of Theorem
\ref{thmGreW} are satisfied. Then we have for
every smooth function $f(r)$ with $f'(r) \leq
0\,\,\textrm{for all}\,\,\, r$, (respectively
$f'(r) \geq 0\,\,\textrm{for all}\,\,\, r$):
\begin{equation} \label{eqLap1}
\begin{aligned}
\Delta^{P}(f \circ r) \, \geq (\leq) \, &\left(\,
f''(r) - f'(r)\eta_{w}(r) \, \right)
 \Vert \nabla^{P} r \Vert^{2} \\ &+ mf'(r) \left(\, \eta_{w}(r) +
\langle \, \nabla^{N}r, \, H_{P}  \, \rangle  \,
\right)  \quad ,
\end{aligned}
\end{equation}
where $H_{P}$ denotes the mean curvature vector
of $P$ in $N$.
\end{corollary}

%%%%%%%%%%%%%%%%%%%%%%%%%%%%%%%%%%%%%%%%%%%%%%%%%%%%%%%%%%%%%%%%%%%%%%%%%%%%%%%%
%%%%%%%%%%%%%%%%%%%%%%%%%%%%%%%%%%%%%%%%%%%%%%%%%%%%%%%%%%%%%%%%%%%%%%%%
%%%%%%%%%%%%%%%%%%%%%%%%%%%%%%%%%%%%%%%%%%%%%%%%%%%%%%%%%%%%%%%%%%%%%%%%%%
%%%%%%%%%%%%%%%%%%%%%%%%%%%%%%%%%%%%%%%%%%%%%%%%%%%%%%%%%%%%%%%%%%%%%%%%%%

%%%%%%%%%%%%%%%%%%%%%%%%%%%%%%%%%%%%%%%%%%%%%%%%%%%%%%%%%%%%%%%%%%%%%%%%%%
%%%%%%%%%%%%%%%%%%%%%%%%%%%%%%%%%%%%%%%%%%%%%%%%%%%%%%%%%%%%%%%%%%%%%%%%%%
%%%%%%%%%%%%%%%%%%%%%%%%%%%%%%%%%%%%%%%%%%%%%%%%%%%%%%%%%%%%%%%%%%%%%%%%%%
%%%%%%%%%%%%%%%%%%%%%%%%%%%%%%%%%%%%%%%%%%%%%%%%%%%%%%%%%%%%%%%%%%%%%%%%%%

\subsection{The Isoperimetric Comparison Space}
\label{secIsopCompSpace} Given the tangency and convexity bounding
functions $g(r)$, $h(r)$ and the ambient
curvature controller function $w(r)$ we construct
a new model space $C^{\,m}_{w, g, h}\,$, which
eventually will serve as the precise comparison
space for the isoperimetric quotients of
extrinsic balls in $P$.

\begin{proposition} \label{propMonster}
For the given smooth functions $w(r)$, $g(r)$, and $h(r)$ defined above on the closed
interval $[\,0, R\,]$, we consider the auxiliary function
$\Lambda(r)$ for $r \in \,\,]\,0, R\,]\,$, which is independent of $R$ and defined via the following equation.
(We without lack of generality that $R > 1$ which, if needed, can be obtained by scaling the metric of $N$ and thence $P$  by a constant.)
\begin{equation} \label{eqLambdaNew}
\Lambda(r)w(r)g(r) \, = \,
T\,\exp\left(-\int_{r}^{\,1}\,\frac{m}{g^{2}(t)}\,\left(\eta_{w}(t)
- h(t)\right)\, dt\right),
\end{equation}
where $T$ is a positive constant - to be fixed shortly.
Then there exists a unique smooth extension of $\Lambda(r)$ to the closed interval
$[\,0, R\,]$ with $\Lambda(0) \, = \, 0$. Moreover, the constant $T$ can be chosen so that
\begin{equation} \label{eqTR}
\frac{d}{dr}_{|_{r=0}}\left(\Lambda^{\frac{1}{m-1}}(r)\right)
= 1 \quad .
\end{equation}
\end{proposition}
\begin{proof}
Concerning the definition of $\Lambda$, the only problem is with the smooth extension to $r\,=\,0$.
We first observe that the integrand of the
right hand side of (\ref{eqLambdaNew}) can be
expressed as follows - upon observing that
$w(0)=0$, $w'(0) = 1$, and $g(0)=1$:
\begin{equation} \label{eqZeta}
\frac{m}{g^{2}(t)}\,\left(\frac{w'(t)}{w(t)} -
h(t)\right)\, = \, \frac{m}{t} + \zeta(t)
\end{equation}
for a smooth, hence also bounded, function $\zeta(t)$ defined
on the closed interval $[\,0, R\,]\,$.

Indeed, since $w(t)$ is assumed smooth in $[\,0, R\,]$ it means in particular,
that the derivatives $w^{(n)}(0) \, = \, \lim_{t \to 0} w^{(n)}(t)$ are well
defined for every 'derivation' number $n$.
A repeated application of L'Hospital's rule together with the facts that  $w(0)=0$, $w'(0) = 1$ then shows,
that the function $\alpha(t) \, = \, w(t)/t$ is also a well defined
bounded smooth positive function on the closed interval $[\,0, R\,]$. In particular
$\alpha(0) \, = \, \lim_{t \to 0} (w(t)/t) \, = \, 1$. Then we have
\begin{equation}
\beta(t) \, = \, \eta_{w}(t) - \frac{1}{t} \, = \, \frac{d}{dt}\left(\log\frac{w(t)}{t}\right) \, = \, \frac{\alpha'(t)}{\alpha(t)}
\end{equation}
is a smooth bounded function on $[\,0, R\,]$.
Yet another repeated application of L'Hospital's rule
and the fact that  $g(0)=1$, shows that
the function
\begin{equation}
\gamma(t) \, = \, \frac{g^{2}(t)}{t} - \frac{1}{t}
\end{equation}
is also a smooth bounded function on $[\,0, R\,]$.

The integrand of the
right hand side of (\ref{eqLambdaNew}) can then be
expressed as follows:
\begin{equation} \label{eqZetaB}
\begin{aligned}
\frac{m}{g^{2}(t)}\,\left(\eta_{w}(t) -
h(t)\right)\, &= \, \frac{m}{g^{2}(t)}\,\left( \frac{1}{t} + \beta(t) -
h(t)\right)\\
&= \, \frac{m}{g^{2}(t)}\,\left( \frac{g^{2}(t)}{t} - \gamma(t) + \beta(t) -
h(t)\right) \\
&= \, \frac{m}{t}+\zeta(t)\quad ,
\end{aligned}
\end{equation}
where
$\zeta(t)$ is a smooth bounded function on $[\,0, R\,]$ as claimed in (\ref{eqZeta}).

Equation
(\ref{eqLambdaNew}) therefore reads for all $r
\in [\,0, R\,]\,$:
\begin{equation}
\Lambda(r)w(r)g(r)\, = \,
T\,r^{m}\exp\left(-\int_{r}^{\,1}\zeta(t)\,dt\right)
\quad.
\end{equation}
Hence applying once more the fact that $w(r) \, = \, r\alpha(r)$ we get:
\begin{equation}
\Lambda(r) \, = \, T\,r^{m-1}G(r)\quad,
\end{equation}
where $G(r)$ is a smooth, bounded, positive function.
Since $G(r)$ is smooth and positive the $(m-1)$'st root of that function is also smooth and positive
throughout $[\,0, R\,]$. In other words:
\begin{equation}
\Lambda^{\frac{1}{m-1}}(r) \, = \, r \left(T\,G(r)\right)^{\frac{1}{m-1}}\quad
\end{equation}
is a smooth, positive, bounded function on $[\,0, R\,]$ with
$\Lambda^{\frac{1}{m-1}}(0) \,=\,0$, and the choice $T\,=\, 1/G(0)$ finally gives
\begin{equation} \label{eqTRverif}
\frac{d}{dr}_{|_{r=0}}\left(\Lambda^{\frac{1}{m-1}}(r)\right)
= 1
\end{equation}
as needed.
\end{proof}

The following observation is a direct consequence of the construction of $\Lambda(r)$:

\begin{lemma} \label{lemLambdaDiffeq}
The function $\Lambda(r)$ satisfies the following differential equation:
\begin{equation} \label{eqLambdaDiffeq}
\begin{aligned}
\frac{d}{dr}\,\Lambda(r)w(r)g(r) \, &= \, \Lambda(r)w(r)g(r)\left(\frac{m}{g^{2}(r)}\left(\eta_{w}(r) - h(r) \right)\right) \\
&= \, m\,\frac{\Lambda(r)}{g(r)}\left(w'(r) - h(r)w(r) \right)
\end{aligned}
\end{equation}
\end{lemma}

\begin{remark} \label{remOnLambda}
In passing we also note that the function $\Lambda(r)$ defined as above
is the essential ingredient in the solution to a Dirichlet--Poisson  problem
(\ref{eqPoisson1}) to be considered below and is in this sense of instrumental
importance for obtaining our main isoperimetric inequality  in Theorem \ref{thmIsopGeneral}.
\end{remark}

A 'stretching' function $s$ is defined as follows
\begin{equation}
s(r) \, = \, \int_{0}^{r}\,\frac{1}{g(t)} \, dt \quad .
\end{equation}
It  has a well-defined inverse $r(s)$ for $s
\in [\,0, s(R)\,]$ with derivative $r'(s) \, = \, g(r(s))$.
In particular $r'(0)\, = \, g(0) \, = \, 1$.\\

With these concepts and key features of the auxiliary function $\Lambda(r)$ we are now able to
define the carrier of the isoperimetric quotients to be compared with the isoperimetric quotients
of the extrinsic balls $D_{R}$:

\begin{definition} \label{defCspace}
The {\em{isoperimetric comparison space}}
$C^{\,m}_{w, g, h}\,$ is the $W-$model space with
base interval $B\,= \, [\,0, s(R)\,]$ and warping
function $W(s)$ defined by
\begin{equation}
W(s) \, = \, \Lambda^{\frac{1}{m-1}}(r(s)) \quad .
\end{equation}
\end {definition}

We observe firstly, that in spite of its
relatively complicated construction, $C^{\,m}_{w,
g, h}\,$ is indeed a model space with a well
defined pole $p_{W}$ at $s = 0$: $W(s) \geq 0$
for all $s$ and $W(s)$ is only $0$ at $s=0$,
where also, because of the explicit construction
in Proposition \ref{propMonster} and equation (\ref{eqTR}):  $W'(0)\, = \, 1\,$.

Secondly it should not be forgotten, that the
spaces $C^{\,m}_{w, g, h}\,$ are specially tailor
made to facilitate the proofs of the
isoperimetric inequalities, that we are about to
develop in section \ref{secMainIsopRes} as well as the
explicit capacity comparison result in section
\ref{secCapAnalysis}.

In order for this to work out we need one further
particular property to be satisfied by these
comparison spaces:

Of course any given $C_{w,g,h}^{m}$ inherits all
its properties from the bounding functions $w$,
$g$, and $h$ from which it is molded in the first
place. Concerning the associated volume growth
properties we note the following expressions for
the isoperimetric quotient function:

\begin{proposition} \label{propIsopFct}
Let $B_{s}^{W}(p_{W})$ denote the metric ball of
radius $s$ centered at $p_{W}$ in
$C^{m}_{w,g,h}$. Then the corresponding
isoperimetric quotient function is
\begin{equation} \label{eqIsopFunction}
\begin{aligned}
q_{W}(s) \, &= \,
\frac{\Vol(B_{s}^{W}(p_{W}))}{\Vol(\partial
B_{s}^{W}(p_{W}))} \, \\ &= \,
\frac{\int_{0}^{s}\,W^{m-1}(t)\,dt}{W^{m-1}(s)}\,
\\ &= \,
\frac{\int_{0}^{r(s)}\,\frac{\Lambda(u)}{g(u)}\,du}{\Lambda(r(s))}
\quad .
\end{aligned}
\end{equation}
\end{proposition}

%%%%%%%%%%%%%%%%%%%%%%%%%%%%%%%%%%%%%%%%%%%%%%%%%%%%%%%%%%%%%%%%%%%%%%%%%%
%%%%%%%%%%%%%%%%%%%%%%%%%%%%%%%%%%%%%%%%%%%%%%%%%%%%%%%%%%%%%%%%%%%%%%%%%%
%%%%%%%%%%%%%%%%%%%%%%%%%%%%%%%%%%%%%%%%%%%%%%%%%%%%%%%%%%%%%%%%%%%%%%%%%%
%%%%%%%%%%%%%%%%%%%%%%%%%%%%%%%%%%%%%%%%%%%%%%%%%%%%%%%%%%%%%%%%%%%%%%%%%%

\subsection{A Balance Condition} \label{secBalanceCond}
\begin{definition} \label{defBalCond}
The model space $M_{W}^{m} \, = \, C_{w, g, h}^{m}$ is
{\em{$w-$balanced from below}} (with respect to the intermediary model space $M_{w}^{m}$) if
the following holds for all $r \in \, [\,0, R\,]$, resp. all $s \in \, [\,0, s(R)\,]$:
\begin{equation} \label{eqDefBalCond}
m\,q_{W}(s)\left(\eta_{w}(r(s)) - h(r(s)) \right) \, \geq g(r(s)) \quad .
\end{equation}
\end{definition}

\begin{remark}
In particular the $w$-balance condition implies that
\begin{equation} \label{eqEtaVSh}
\eta_{w}(r) \, - h(r) \, > \, 0 \quad .
\end{equation}

The above definition of a $w-$balance condition from below is an extension of the balance
condition from below as applied in \cite{MP4}.
The condition there reads as follows and is obtained from (\ref{eqDefBalCond}) precisely when
$g(r) \, = \, 1$ and $h(r) \, = \, 0$ for all $r
\in [\,0, R\,]\,$ so that $r(s) \, =\, s$, $W(s)\, = \, w(r)\,$:
\begin{equation}
m\,q_{w}(r)\eta_{w}(r)\, \geq 1 \quad .
\end{equation}

\end{remark}

\begin{lemma} \label{lemBalExpand}
The balance condition (\ref{eqDefBalCond}) is equivalent to the following inequality:
\begin{equation} \label{eqBalanceExpand}
m\left(\int_{0}^{r}\,\frac{\Lambda(t)}{g(t)}\,dt\right)\left(w'(r)
-  h(r)w(r) \right) \, \geq \, \Lambda(r)w(r)g(r) \quad .
\end{equation}
Moreover, the balance condition (\ref{eqDefBalCond}), hence (\ref{eqBalanceExpand}), is {\em{implied by}} the following (thence stronger) inequality, which
does not involve the tangency bounding function
$g(r)$:
\begin{equation}\label{eqBalB}
w''(r) - w'(r)h(r) - w(r)h'(r) \, \geq \, 0 \quad
.
\end{equation}
\end{lemma}
\begin{proof}
A direct differentiation as in (\ref{eqIsopFunction}) but with respect to $r$ amounts to:
\begin{equation*}
\begin{aligned}
&\frac{d}{dr}\left(\frac{q_{W}(s(r))}{g(r)w(r)}\right) \, = \,
\frac{d}{dr}\left(\frac{\int_{0}^{r}\,\frac{\Lambda(u)}{g(u)}\,du}{\Lambda(r)g(r)w(r)}\right) \\
&= \, \frac{\Lambda(r)}{g^{2}(r)\Lambda(r)w(r)} - \frac{\frac{d}{dr}\left(\Lambda(r)g(r)w(r)\right)\int_{0}^{r}\,\frac{\Lambda(u)}{g(u)}\,du}{\Lambda^{2}(r)g^{2}(r)w^{2}(r)} \\
&= \, \frac{1}{\Lambda(r)g(r)w(r)}\left(\frac{\Lambda(r)}{g(r)} - \left(\frac{m}{g^{2}(r)}(\eta_{w}-h(r))\right)\int_{0}^{r}\,\frac{\Lambda(u)}{g(u)}\,du \right) \\
&= \,  \frac{1}{\Lambda(r)g^{2}(r)w^{2}(r)}\left(\Lambda(r)w(r)g(r) - m\left(\int_{0}^{r}\,\frac{\Lambda(t)}{g(t)}\,dt\right)\left(w'(r)
-  h(r)w(r) \right)    \right) \quad ,
\end{aligned}
\end{equation*}
which shows that (\ref{eqDefBalCond}) and (\ref{eqBalanceExpand}) are equivalent. \\

Now let $G(r)$ denote the left hand side of the balance condition inequality (\ref{eqBalanceExpand}),
and let $F(r)$ denote the right
hand side:
\begin{equation*}
G(r) \, = \, m\,\left(\int_{0}^{r}\,\frac{\Lambda(t)}{g(t)}\,dt\right)\left(w'(r)
-  h(r)w(r) \right) \, \geq \, \Lambda(r)w(r)g(r) \, = \, F(r)
\quad .
\end{equation*}

 Then $G(0)\, = \, F(0)=0$ and, moreover
\begin{equation} \label{eqDifIneq}
\begin{aligned}
G'(r) \, &= \,
m\frac{\Lambda(r)}{g(r)}\left(w'(r) - h(r)w(r)
\right) \\ &+ m\,\left(\int_{0}^{r}\,
\frac{\Lambda(r)}{g(r)} \,dt \right)\left(w''(r)
- w'(r)h(r) - w(r)h'(r) \right)  \\  &\geq \,
 m\frac{\Lambda(r)}{g(r)}\left(w'(r)
- h(r)w(r) \right) \\ &= \, F'(r) \quad,
\end{aligned}
\end{equation}
where we have used Lemma \ref{lemLambdaDiffeq} (for
$F'(r)$) and the balance condition
(\ref{eqBalB}). It follows that $G(r) \, \geq \,
F(r)$ for all $r \in [\,0, R\,]\,$, and this proves
the Lemma.
\end{proof}

%%%%%%%%%%%%%%%%%%%%%%%%%%%%%%%%%%%%%%%%%%%%%%%%%%%%%%%%%%%%%%%%%%%%%%%%%%%%%%
%%%%%%%%%%%%%%%%%%%%%%%%%%%%%%%%%%%%%%%%%%%%%%%%%%%%%%%%%%%%%%%%%%%%%%%%%%%%%%

\subsection{Balance on the Edge}

\begin{lemma} \label{lemBalanceEqualty}
Equality in the balance condition
(\ref{eqDefBalCond}) is equivalent to equality in the stronger
condition (\ref{eqBalB}) and equivalent to each one of the following identities:
\begin{equation}
\eta_{w}(r) - h(r) \, = \, \frac{1}{w(r)}\quad
\end {equation}
\begin{equation} \label{eqBalEqB}
\Lambda(r)w(r)g(r) \, = \,  m\int_{0}^{r}\,\frac{\Lambda(t)}{g(t)}\,dt \quad .
\end {equation}
\begin{equation}
q_{W}(s) \, = \, \frac{1}{m}w(r(s))g(r(s)) \quad .
\end{equation}
\end{lemma}
\begin{proof}
From equality in (\ref{eqBalB}) we get
\begin{equation}
w'(r) \, = \, w(r)h(r)+c
\end{equation}
for some constant $c$ which must then be $c=1$
since $w'(0) = 1\,$.
Then $w'(r) - w(r)h(r) \, = \, 0$ gives identity in equation (\ref{eqDifIneq}) as well and vice versa.
\end{proof}

\begin{remark}
Special cases of equality in the balance
condition are obtained by $h(r)\, = \, C$ and
$w(r)\,= \,r$. This corresponds to the situation
considered in section \ref{secExampSurfOfRev} -
where we analyzed radially mean $0-$convex surfaces
in Euclidean $3-$space.
\end{remark}

%%%%%%%%%%%%%%%%%%%%%%%%%%%%%%%%%%%%%%%%%%%%%%%%%%%%%%%%%%%%%%%%%%%%%%%%%%
%%%%%%%%%%%%%%%%%%%%%%%%%%%%%%%%%%%%%%%%%%%%%%%%%%%%%%%%%%%%%%%%%%%%%%%%%%
%%%%%%%%%%%%%%%%%%%%%%%%%%%%%%%%%%%%%%%%%%%%%%%%%%%%%%%%%%%%%%%%%%%%%%%%%%
%%%%%%%%%%%%%%%%%%%%%%%%%%%%%%%%%%%%%%%%%%%%%%%%%%%%%%%%%%%%%%%%%%%%%%%%%%

\section{The Isoperimetric Comparison Constellation} \label{secIsopConstellation}

The intermediate observations considered above together with the
previously introduced bounds on radial curvature
and tangency now constitute the notion of a
comparison constellation (for the isoperimetric
inequality) as follows.

\begin{definition}\label{defConstellatNew}
Let $N^{n}$ denote a Riemannian manifold with a
pole $p$ and distance function $r \, = \, r(x) \,
= \, \dist_{N}(p, x)$. Let $P^{m}$ denote an
unbounded complete and closed submanifold in
$N^{n}$. Suppose $p \in P^{m}$, and suppose that the following
conditions are
satisfied for all $x \in P^{m}$ with $r(x) \in
[\,0, R\,]\,$:
\begin{enumerate}[(a)]
\item The $p$-radial sectional curvatures of $N$ are bounded from below
by the $p_{w}$-radial sectional curvatures of
the $w-$model space $M_{w}^{m}$:
$$
\mathcal{K}(\sigma_{x}) \, \geq \,
-\frac{w''(r(x))}{w(r(x))} \quad .
$$

\item The $p$-radial mean curvature of $P$ is bounded from below by
a positive smooth radial function $h(r)$:
$$
\mathcal{C}(x)  \geq h(r(x)) \quad.
$$

\item The $p$-radial tangency of $P$ is
bounded from below by a smooth radial function
$g(r)$:
$$
\mathcal{T}(x) \, \geq \, g(r(x)) \, > \, 0 \quad.
$$
\end{enumerate}
Let $C_{w,g,h}^{m}$ denote the $W$-model with the
specific warping function $W: \pi(C_{w,g,h}^{m})
\to \mathbb{R}_{+}$ which is constructed above in
Definition \ref{defCspace} via $w$, $g$, and $h$.
Then the triple $\{ N^{n}, P^{m}, C_{w,g,h}^{m}
\}$ is called an {\em{isoperimetric comparison
constellation}} on the interval $[\,0, R\,]\,$.
\end{definition}

%%%%%%%%%%%%%%%%%%%%%%%%%%%%%%%%%%%%%%%%%%%%%%%%%%%%%%%%%%%%%%%%%%%%%%%%%%%%%%%%
%%%%%%%%%%%%%%%%%%%%%%%%%%%%%%%%%%%%%%%%%%%%%%%%%%%%%%%%%%%%%%%%%%%%%%%%
%%%%%%%%%%%%%%%%%%%%%%%%%%%%%%%%%%%%%%%%%%%%%%%%%%%%%%%%%%%%%%%%%%%%%%%%%%
%%%%%%%%%%%%%%%%%%%%%%%%%%%%%%%%%%%%%%%%%%%%%%%%%%%%%%%%%%%%%%%%%%%%%%%%%%
%%%%%%%%%%%%%%%%%%%%%%%%%%%%%%%%%%%%%%%%%%%%%%%%%%%%%%%%%%%%%%%%%%%%%%%%%%
%%%%%%%%%%%%%%%%%%%%%%%%%%%%%%%%%%%%%%%%%%%%%%%%%%%%%%%%%%%%%%%%%%%%%%%%%%
%%%%%%%%%%%%%%%%%%%%%%%%%%%%%%%%%%%%%%%%%%%%%%%%%%%%%%%%%%%%%%%%%%%%%%%%%%
%%%%%%%%%%%%%%%%%%%%%%%%%%%%%%%%%%%%%%%%%%%%%%%%%%%%%%%%%%%%%%%%%%%%%%%%%%

\section{Main isoperimetric results}
\label{secMainIsopRes}

In this section we find upper bounds for the isoperimetric
quotient defined as the volume of the extrinsic
sphere divided by the volume of the extrinsic
ball, in the setting given by the comparison
constellations defined in Definition \ref{defConstellatNew}:

\begin{theorem} \label{thmIsopGeneral}
We consider an isoperimetric comparison
constellation $\{ N^{n}, P^{m}, C_{w,g,h}^{m}
\}$ on the interval $[\,0, R\,]$.
 Suppose further that the comparison space
 $\, C_{w,g,h}^{m}\,$ is $w-$balanced from
 below in the sense of Definition (\ref{defBalCond}).
Then
\begin{equation} \label{eqIsopGeneralA}
\frac{\Vol(\partial D_{R})}{\Vol(D_{R})} \leq
\frac{\Vol(\partial
B^{W}_{s(R)})}{\Vol(B^{W}_{s(R)})} \, \leq \,
\frac{m}{g(R)}\left(\eta_{w}(R) -
h(R)\right)\quad .
\end{equation}
 If the comparison space
 $C^{m}_{w,g,h}$ satisfies the balance condition with
 equality in (\ref{eqDefBalCond}) (or equivalently in (\ref{eqBalB})) for all $r \in [\,0, R\,]\,$, then
\begin{equation} \label{eqIsopGeneralB}
\frac{\Vol(\partial D_{R})}{\Vol(D_{R})} \leq
\frac{\Vol(\partial
B^{W}_{s(R)})}{\Vol(B^{W}_{s(R)})}\, = \,
\frac{m}{w(R)g(R)}\quad .
\end{equation}
\end{theorem}

\begin{remark}
We first note here, that Theorem \ref{thmAisop} follows readily from this result, since the $w-$balance condition
holds (with equality) in the simple situation of Theorem \ref{thmAisop}.
\end{remark}

\begin{proof}
We define a second order differential operator
$\LL$ on functions $f$ of one real variable as
follows:
\begin{equation}
\LL f(r) \, = \, f''(r)\,g^{2}(r) +
f'(r)\left((m-g^{2}(r))\,\eta_{w}(r) - m\,h(r)
\right) \quad ,
\end{equation}
and consider the smooth solution $\psi(r)$ to the
following Dirichlet--Poisson problem:
\begin{equation}   \label{eqPoisson1}
\begin{aligned}
\LL \psi(r) &= -1\,\,\,\,\text{on}\,\,\,\, [\,0, R\,]\\
\psi(R) &=0
\end{aligned}
\end{equation}
The ODE problem is equivalent to the following:
\begin{equation}
\psi''(r) +
\psi'(r)\left(-\eta_{w}(r) + \frac{m}{g^{2}(r)}\left(\eta_{w}(r) - \,h(r)\right)\right)\, = \, -\frac{1}{g^{2}(r)}\quad .
\end{equation}
The solution is constructed via the auxiliary
function  $\Lambda(r)$ from (\ref{eqLambdaNew})
as follows:
\begin{equation}
\begin{aligned}
\psi'(r) \, &= \, \Gamma(r) \, \\ &= \,
\exp(-\mathcal{P}(r))\,\int_{0}^{r}\,\exp(\mathcal{P}(t))\left(-\frac{1}{g^{2}(t)}\right)\,dt \quad,
\end{aligned}
\end{equation}
where the auxiliary function $\mathcal{P}$ is defined as follows, assuming again without lack of generality, that
$R \, > \, 1$, so that the domain of definition of all involved functions contains $[\,0, 1\,]\,$:
\begin{equation}
\begin{aligned}
\mathcal{P}(r) \, &= \, \int_{r}^{1}\, \left(\eta_{w}(t) - \frac{m}{g^{2}(t)}\left(\eta_{w}(t) - \,h(t)\right)\right)\, dt \\
&= \, \log\left(\frac{w(1)}{w(r)}\right) - \int_{r}^{\,1}\, \frac{m}{g^{2}(t)}\left(\eta_{w}(t) - \,h(t)\right)\, dt
\quad .
\end{aligned}
\end{equation}
Hence
\begin{equation}
\begin{aligned}
&\exp(\mathcal{P}(t))\left(\frac{1}{g^{2}(t)}\right) \, = \, \\ &\frac{w(1)}{w(t)\,g^{2}(t)} \exp\left(-\int_{t}^{\,1}\,\frac{m}{g^{2}(u)}\,\left(\eta_{w}(u)
- h(u)\right)\, du\right) \\
&= \,  \frac{w(1)}{w(t)\,g^{2}(t)\,T} \Lambda(t)w(t)g(t)\, = \, \left(\frac{w(1)}{T}\right)\,\left(\frac{\Lambda(t)}{g(t)}\right)  \quad.
\end{aligned}
\end{equation}
From this latter expression we then have as well
\begin{equation}
\exp(-\mathcal{P}(r)) \, = \,  \left(\frac{T}{w(1)}\right)\,\left(\frac{1}{g(r)\Lambda(r)}\right) \quad ,
\end{equation}
so that
\begin{equation}\label{eqWellDef}
\begin{aligned}
\psi'(r) \, &= \,
\frac{-1}{g(r)\,\Lambda(r)}\,\int_{0}^{r}\,\frac{\Lambda(t)}{g(t)}\,dt
\, \\ &= -
\frac{\Vol(B^{W}_{s(r)})}{g(r)\Vol(\partial
B^{W}_{s(r)})} \\ &= -
\frac{q_{W}(s(r))}{g(r)}\quad .
\end{aligned}
\end{equation}
At this instance we must observe, that it follows clearly from this latter expression for $\psi'(r)$ that
the above formal, yet standard, method of solving (\ref{eqPoisson1}) indeed does produce a smooth solution on
$[\,0, R \,]$ with $\psi'(0) \, = \, 0$.
Then we have:
\begin{equation} \label{eqMeanExitMonster}
\begin{aligned}
\psi(r) \,&= \,
\int_{r}^{R}\frac{1}{g(u)\,\Lambda(u)}\,\left(\int_{0}^{u}\,\frac{\Lambda(t)}{g(t)}\,dt\,\right)\,du
\\ &= \int_{r}^{R} \frac{q_{W}(s(u))}{g(u)} \, du
\\ & = \int_{s(r)}^{s(R)} \,q_{W}(t) \, dt
\quad .
\end{aligned}
\end{equation}

We now show that - because of the balance
condition (\ref{eqBalB}) - the function $\psi(r)$
enjoys the following inequality:
\begin{lemma} \label{lemYellow}
\begin{equation} \label{eqYellow}
\psi''(r) - \psi'(r)\,\eta_{w}(r) \, \geq \, 0 \
\quad .
\end{equation}
\end{lemma}
\begin{proof}[Proof of Lemma]
We must show that
\begin{equation} \label{eqYellowGam}
\Gamma'(r) - \Gamma(r)\eta_{w}(r) \, \geq \, 0
\quad .
\end{equation}
Since
\begin{equation}
\begin{aligned}
\Gamma'(r) \, &= \, \psi''(r) \\&= \,
-\frac{1}{g^{2}(r)}\left(1 +
\psi'(r)\left((m-g^{2}(r))\,\eta_{w}(r) -
m \, h(r) \right) \right) \\
&= \, -\frac{1}{g^{2}(r)}\left(1 +
\Gamma(r)\left((m-g^{2}(r))\,\eta_{w}(r) -
m\,h(r) \right)\right)
\end{aligned}
\end{equation}
we have
\begin{equation}
\Gamma'(r) - \Gamma(r)\eta_{w}(r) \, = \,
-\frac{1}{g^{2}(r)}\left(1 +
m\Gamma(r)\left(\eta_{w}(r) -  h(r) \right)
\right) \quad .
\end{equation}
Therefore equation (\ref{eqYellowGam}) is
equivalent to each one of the following
inequalities, observing that $\Gamma(r) < 0$ for
all $r \in \, ]0, R\,]$:
\begin{equation} \label{eqBal1}
m\Gamma(r)\left(\eta_{w}(r) -  h(r) \right) \,
\leq \, -1 \quad ,
\end{equation}
\begin{equation} \label{eqBal2}
\frac{m}{g(r)\,\Lambda(r)}\,\left(\int_{0}^{r}\,
\frac{\Lambda(t)}{g(t)}\,dt\right)\left(\frac{w'(r)}{w(r)}
-  h(r) \right) \, \geq \, 1 \quad ,
\end{equation}
\begin{equation} \label{eqBal3}
m\,\left(\int_{0}^{r}\,\frac{\Lambda(t)}{g(t)}\,dt\right)\left(w'(r)
-  h(r)w(r) \right) \, \geq \, \Lambda(r)w(r)g(r)
\quad ,
\end{equation}
where the latter inequality is equivalent to (\ref{eqDefBalCond}) via Lemma
\ref{lemBalExpand}.
\end{proof}

With Lemma \ref{lemYellow} in hand we continue the proof of Theorem \ref{thmIsopGeneral}.
Applying the Laplace inequality (\ref{eqLap1})
for the function $\psi(r)$ transplanted into
$P^{m}$ in $N^{n}$ now gives the following
comparison
\begin{equation} \label{eqMeanExit}
\begin{aligned}
\Delta^{P}\psi(r(x)) \, &\geq \,
\left(\psi''(r(x))- \psi'(r(x))\eta_{w}(r(x))\right)\,g^{2}(r(x)) \\
& \phantom{mm}+ m\psi'(r(x))\left( \eta_{w}(r(x))
- h(r(x))\right)\, \\ &= \, \LL \psi(r(x)) \, \\
&= \, -1 \, \\ &= \, \Delta^{P}E(x) \quad ,
\end{aligned}
\end{equation}
where $E(x)$ is the mean exit time function for
the extrinsic ball $D_{R}\,$, with
$E_{|_{\partial D_{R}}}\,=\,0$. The maximum
principle then applies and gives:
\begin{equation} \label{eqMeanExitComp}
E(x) \, \geq \, \psi(r(x))\,\, , \,\, \textrm{for
all} \, \, x \in D_{R} \quad .
\end{equation}

Applying the divergence theorem, taking the unit
normal
    to $\partial D_R$ as $\frac{\nabla^{P} r}{\Vert \nabla^{P} r\Vert}$,
we get
    \begin{equation}\label{eqVolFrac}
    \begin{aligned} \Vol(D_{R}) &= \int_{D_{R}} -\Delta^{P} u(x)\, d\mu \\ &\geq
    \int_{D_{R}} -\Delta^{P} \psi(r(x))\, d\mu \\ &= -\int_{D_{R}} \Div(\nabla^{P} \psi(r(x)))\, d\mu \\ &=
    -\int_{\partial D_{R}}<\nabla^{P} \psi(r(x)), \frac{\nabla^{P} r(x)}
    {\Vert\nabla^{P} r\Vert}>\,d\nu \\ &=
    -\Gamma(R)\int_{\partial D_{R}}\Vert\nabla^P r\Vert\, d\nu \\ & \geq
    -\Gamma(R)\,g(R)\Vol(\partial D_{R}) \quad ,
    \end{aligned}
    \end{equation}
    which shows the isoperimetric inequality (\ref{eqIsopGeneralA}).\\

We finally consider the case of equality in the balance condition, which implies in particular that
(\ref{eqBalEqB}) holds true. Using this we have immediately, as claimed:
\begin{equation}
\frac{\Vol(\partial
B^{W}_{s(R)})}{\Vol(B^{W}_{s(R)})}\, = \,\frac{\Lambda(R)}{\int_{0}^{R}\frac{\Lambda(t)}{g(t)}\, dt}\, = \,
\frac{m}{g(R)\,w(R)}\quad .
\end{equation}
\end{proof}

%%%%%%%%%%%%%%%%%%%%%%%%%%%%%%%%%%%%%%%%%%%%%%%%%%%%%%%%%%%%%%%%%%%%%%%%%
%%%%%%%%%%%%%%%%%%%%%%%%%%%%%%%%%%%%%%%%%%%%%%%%%%%%%%%%%%%%%%%%%%%%%%%%%
\section{Consequences} \label{secConseq}
We first observe the following volume comparison in consequence
of Theorem \ref{thmIsopGeneral}, inequality (\ref{eqIsopGeneralA}):
\begin{corollary} \label{corVolBound}
Let $\{ N^{n}, P^{m}, C_{w,g,h}^{m} \}$ be a
comparison constellation on the interval $[\,0,
R\,]\,$, as in Theorem \ref{thmIsopGeneral}. Then
\begin{equation}
\Vol(D_{r}) \, \leq \, \Vol(B_{s(r)}^{W})
\,\,\,\,\,\,\textrm{for every} \,\,\, r \in [\,0,
R\,] \,\, .
\end{equation}
\end{corollary}
\begin{proof}
Let $\mathcal{G}(r)$ denote the following function
\begin{equation}
\mathcal{G}(r)\, = \, \log
\left(\frac{\Vol(D_{r})}{\Vol(B_{s(r)}^{W})}
\right) \quad .
\end{equation}
Since $W(s)$ is a warping function for an
$m-$dimensional model space we  have $\mathcal{G}(0)\, = \,
\lim_{r \to 0} \mathcal{G}(r) \, = \, 0$ and furthermore
\begin{equation}
\mathcal{G}'(r) \, = \,
\left(\frac{\frac{\partial}{\partial
r}\Vol(D_{r})}{\Vol(D_{r})}\right) -
\left(\frac{\frac{\partial}{\partial
s}\Vol(B^{W}_{s})}{g(r)\Vol(B^{W}_{s(r)})}\right)
\quad .
\end{equation}
From the co-area formula (see below) we have:
\begin{equation}
g(r)\frac{\partial}{\partial r}\Vol(D_{r}) \,
\leq \, \Vol(\partial D_{r}) \quad ,
\end{equation}
so that the isoperimetric inequality
(\ref{eqIsopGeneralA}) gives:
\begin{equation}
\mathcal{G}'(r) \, \leq \,
\frac{1}{g(r)}\left(\frac{\Vol(\partial
D_{r})}{\Vol(D_{r})} - \frac{\Vol(\partial
B^{W}_{s(r)})}{\Vol(B^{W}_{s(r)})}\right) \, \leq
\, 0 \quad .
\end{equation}
In consequence we therefore have $\mathcal{G}(r) \, \leq \,
\mathcal{G}(0)\, = \, 0$, or equivalently:
\begin{equation}
\Vol(D_{r}) \, \leq \, \Vol(B_{s(r)}^{W})
\,\,\,\,\,\,\textrm{for every} \,\,\, r \in [\,0,
R\,] \quad .
\end{equation}
\end{proof}

%%%%%%%%%%%%%%%%%%%%%%%%%%%%%%%%%%%%%%%%%%%%%%%%%%%%%%%%%%%%%%%%%%%%%%%%%%%%

\begin{proposition} \label{propCoArea}
Let $\{ N^{n}, P^{m}, C_{w,g,h}^{m} \}$ be a
comparison constellation on the interval $[\,0,
R\,]\,$, as in Theorem \ref{thmIsopGeneral}. Then
we have:
\begin{equation}
g(r)\frac{\partial}{\partial r}\Vol(D_{r}) \,
\leq \, \Vol(\partial D_{r}) \quad {\textrm{for
all}}\quad  r \in [\,0, R\,]\,\, .
\end{equation}
\end{proposition}
\begin{proof}
Let $\Psi(x)= \psi(r(x))$ denote the radial mean
exit time function, transplanted into $D_{R}$.

With the notation of \cite{Cha1} we then have
$$\begin{aligned}
\Omega(t) &=\{ x\in P\,|\, \psi(r(x)) > t\}= D_{\psi^{-1}(t) }\\
V(t) &= \Vol(D_{\psi^{-1}(t)})\\
\textrm{and}\,\,\,\,\Sigma(t)&=\partial
D_{\psi^{-1}(t)}\,\,\, \textrm{for all}\,\,\,  t
\in  ]\,0, \psi(0)\,[ \quad .
\end{aligned}
$$

The co-area formula states that
\begin{equation}V'(t)=- \int_{\partial D_{\psi^{-1}(t)}}
\Vert \nabla^P \Psi(x) \,\Vert^{-1} d\sigma_t
 \quad .
\end{equation}

On the other hand, we know that on $D_{R}$ we
have for all $r$:
\begin{equation}
\begin{aligned}
\Vert \nabla^P \Psi \,\Vert \, = \,
-\Gamma(r)\,\Vert \nabla^{P}r\Vert \, \geq
-g(r)\,\Gamma(r) \quad ,
\end{aligned}
\end{equation}
so we then have:
\begin{equation}
-\Vert\nabla^P \Psi \,\Vert ^{-1} \, \geq \,
\frac{1}{g(r)\,\Gamma(r)} \quad .
\end{equation}
Therefore
\begin{equation} \label{eqVP}
\begin{aligned}
V'(t) \, &\geq \,   \int_{\partial
D_{\psi^{-1}(t)}}\,\frac{1}{g(r)\,\Gamma(r)}
d\sigma_t \, \\ &= \, \frac{1}{g(r)\,\Gamma(r)}
\Vol(\partial D_{\psi^{-1}(t)}) \quad .
\end{aligned}
\end{equation}

We define $F(r)= \Vol(D_r)$ and have
\begin{equation}
V(t)=\Vol(D_{\psi^{-1}(t)})= F\circ \psi^{-1}(t)
\quad , \end{equation}
 then
\begin{equation}
\begin{aligned}
V'(t) \, &= \,  F'(\psi^{-1}(t))\,
\frac{d}{dt}\psi^{-1}(t) \, \\ & = \,
\frac{\frac{d}{dr}\Vol(D_{r})}{\frac{d}{dr}\psi(r)}
\\ & = \,
\frac{\frac{d}{dr}\Vol(D_{r})}{\Gamma(r)} \quad.
\end{aligned}
\end{equation}
Since we also know from equation (\ref{eqVP})
that
\begin{equation}
 V'(t) \,\geq \, \frac{\Vol(\partial
D_{r})}{g(r)\,\Gamma(r)} \quad ,
\end{equation}
and since $\Gamma(r) < 0$ on $]\, 0, R[\,$, we
finally get, as claimed:
\begin{equation}
g(r)\frac{d}{dr} \Vol(D_r) \, \leq \,
\Vol(\partial D_{r}) \quad \textrm{for all}
\,\,\, r \in [\,0,R\,] \quad .
\end{equation}

\end{proof}
%%%%%%%%%%%%%%%%%%%%%%%%%%%%%%%%%%%%%%%%%%%%%%%%%%%%%%%%%%%%%%%%%%%%%%%%%%%%

\begin{corollary} \label{corGradientIntBound}
Let $\{ N^{n}, P^{m}, C_{w,g,h}^{m} \}$ be a
comparison constellation on $[\,0, R\,]\,$, as in
Theorem \ref{thmIsopGeneral}. Then
\begin{equation}
\int_{\partial D_{r}}\,\Vert \nabla^{P}r \Vert
\,d\nu \, \leq \, g(r)\,W^{m-1}(s(r))
\,\,\,\textrm{for every} \,\,\, r \in [\,0, R\,]
\quad.
\end{equation}
\end{corollary}
\begin{proof}
From the proof of the isoperimetric inequality,
Theorem \ref{thmIsopGeneral}, we extract the
following inequality from equation
(\ref{eqVolFrac})
\begin{equation}
\int_{\partial D_{\rho}}\,\Vert \nabla^{P}r \Vert
\,d\nu \, \leq \,
\frac{\Vol(D_{\rho})}{-\Gamma(\rho)} \quad.
\end{equation}
From the previous Corollary \ref{corVolBound} we
have
\begin{equation}
\Vol(D_{\rho}) \, \leq \, \Vol(B^{W}_{s(\rho)}) \quad ,
\end{equation}
so that
\begin{equation}
\int_{\partial D_{\rho}}\,\Vert \nabla^{P}r \Vert
\,d\nu \, \leq \,
\frac{\Vol(B^{W}_{s(\rho)})}{-\Gamma(\rho)}\, =
\, g(\rho)\,\Lambda(\rho) \quad.
\end{equation}
\end{proof}

%%%%%%%%%%%%%%%%%%%%%%%%%%%%%%%%%%%%%%%%%%%%%%%%%%%%%%%%%%%%%%%%%%%%%%%%%%%%

%COROLLARY TWO-SIDED

When the submanifold $P^m$ is minimal and the ambient space $N^{n}$ is a Cartan-Hadamard manifold (with sectional curvatures bounded from above by $0$), we get the following two-sided isoperimetric inequality:

\begin{corollary} \label{corTwoSidedIsop}
Let $P^m$ be a minimal submanifold in $N^n$. Let $p \in P^{m}$ be a point which is a pole of $N$. Suppose that the $p$-radial sectional curvatures of $N$ are bounded from above and from below as follows:
\begin{equation}
-\frac{w_2''(r(x))}{w_2(r(x))}\, \leq \,\mathcal{K}(\sigma_{x}) \, \leq \,
-\frac{w_1''(r(x))}{w_1(r(x))} \, \leq \, 0 \quad .
\end{equation}
Then
\begin{equation*}
\frac{\Vol(\partial
B^{w_1}_{R})}{\Vol(B^{w_1}_{R})} \, \leq \, \frac{\Vol(\partial D_{R})}{\Vol(D_{R})} \,  \leq \,
\frac{\Vol(\partial
B^{W_2}_{s(R)})}{\Vol(B^{W_2}_{s(R)})} \, \leq \, \frac{m}{g(R)}\left(\eta_{w_{2}}(R) -
h(R)\right)\, \, ,
\end{equation*}
where $W_{2}(s)$ is the warping function of the comparison model space $C^{m}_{w_{2}, g, 0}$, see Theorem \ref{thmIsopGeneral}.
\end{corollary}
\begin{proof}
Since  $P$ is minimal it is radially mean $0-$convex, so $h\,=\,0$.
The upper isoperimetric bound follows then directly from Theorem \ref{thmIsopGeneral}. Indeed,
$\{ N^{n}, P^{m}, C_{w_{2},g,0}^{m} \}$ is a comparison constellation
with the model comparison space $C_{w_{2},g,0}^{m}$ which is $w_{2}-$balanced from below.
This latter claim follows  because $w_{2}(r)$
satisfies the strong balance condition (\ref{eqBalB}) in view of the assumptions
$-w_{2}''(r(x))/w_{2}(r(x))\, \leq \, 0$ and $h\,=\,0$.
The lower isoperimetric bound follows directly from Theorem B in \cite{MP4}. Also for this to hold we need the corresponding
balance condition from below, which is again satisfied because of the curvature assumption:
$-w_{1}''(r(x))/w_{1}(r(x))\, \leq \, 0$, see \cite[Observation 5.12\,]{MP4}.
\end{proof}

%%%%%%%%%%%%%%%%%%%%%%%%%%%%%%%%%%%%%%%%%%%%%%%%%%%%%%%%%%%%%%%%%%%%%%%%%%%%
%%%%%%%%%%%%%%%%%%%%%%%%%%%%%%%%%%%%%%%%%%%%%%%%%%%%%%%%%%%%%%%%%%%%%%%%%%%%

\section{The intrinsic Viewpoint} \label{secIntrinsic}
In this section we consider the intrinsic versions of the isoperimetric
and volume comparison inequalities (\ref{thmIsopGeneral}) and (\ref{corVolBound}) assuming 
that $P^{m}\, = \, N^{n}$. In this case, the extrinsic distance to the pole $p$ becomes the
intrinsic distance in the ambient manifold, so, for all $r \, > \, 0$ the extrinsic
domains $D_r$ become the geodesic balls $B^N_r$ of $N^n$. We have for all $x \in P$:
\begin{equation}
\begin{aligned}
\nabla^P r(x)\,&=\,\nabla^N r(x) \\
H_P(x)\,&=\,0 \quad .
\end{aligned}
\end{equation}
Thus $\Vert \nabla^P r\Vert \,  =\, 1$, so $g(r(x))\, = \, 1$ and $\mathcal{C}(x)  \, = \,  h(r(x))\, =\, 0$, the {\em stretching} function becomes the identity $s(r) \, = \, r$, $W(s(r))\, = \, w(r)$, and the isoperimetric comparison space $C^m_{w,g,h}$ is reduced to the auxiliary model space $M^m_w$. Applying Proposition \ref{propMonster} we then have
\begin{corollary}\label{corIntrinIsop}
Let $N^n$ denote a complete riemannian manifold with a pole $p$. Suppose that the $p$-radial  sectional curvatures of $N^n$ are bounded from below by the $p_w$-radial sectional curvatures of a $w$-model space $M^n_w$ for all $r >0 $. Then
\begin{equation}\label{eqIntrinIsop}
\frac{\Vol(\partial B^N_{R})}{\Vol(B^N_{R})} \leq
\frac{\Vol(\partial
B^{w}_{R})}{\Vol(B^{w}_{R})} \, \leq \,
n\eta_{w}(R)\quad .
\end{equation}
\end{corollary}
\begin{proof}
The proof follows along the lines of theorem \ref{thmIsopGeneral}. Indeed, since $g(r)\,=\,1$ and $h(r)\,=\,0$, the second order differential operator $L$ agrees with the Laplacian on functions of one real variable defined on the model spaces $M^n_w$, namely,
\begin{equation}
Lf(r)=f''(r)+(n-1)\eta_w(r)f'(r)
\end{equation}
Solving the corresponding problem (\ref{eqPoisson1}) on $[0,R\,]$ under these conditions and transplanting the solution $\psi(r)$ to the geodesic ball $B^N_R$, we can apply the Laplacian comparison analysis using formally $H_P=0$ because $P=N$. Then, we obtain the inequality
\begin{equation}\label{eqMeanExitInt}
\Delta^N\psi(r(x)) \geq -1=\Delta^N E(x) \quad .
\end{equation}
As a consequence of equality $\Vert \nabla^P r\Vert =1$, inequality (\ref{eqYellow})
is obsolete for the statement of inequality (\ref{eqMeanExit}) to hold true.
Hence, we do not need the balance condition as hypothesis, and thence the
conclusion follows as in the proof of Theorem \ref{thmIsopGeneral},
applying the divergence theorem together with the inequality (\ref{eqMeanExitInt}).
\end{proof}

The well-known volume comparison is a straightforward consequence, see \cite{Sa},
and recovered as follows:

\begin{corollary}\label{corVolIneq}
Let $N^n$ denote a complete Riemannian manifold with a pole $p$.
Suppose that the $p$-radial  sectional curvatures of $N^n$ are
bounded from below by the $p_w$-radial sectional curvatures of
a $w$-model space $M^n_w$ for all $r >0$. Then
\begin{equation}
\Vol(B^N_{r}) \, \leq \, \Vol(B_{r}^{w})
\,\,\,\,\,\,\textrm{for every} \,\,\, r \in [\,0,
R\,] \,\, .
\end{equation}
\end{corollary}
\begin{proof}
We follow the lines of the proof of Corollary \ref{corVolBound},
but now letting $\mathcal{G}(r)$ denote the following function:
\begin{equation}
\mathcal{G}(r)\, = \, \log
\left(\frac{\Vol(D_{r})}{\Vol(B_{r}^{w})}
\right) \quad .
\end{equation}
 Taking into account that $\Vert \nabla^P r\Vert =\Vert \nabla^N r\Vert =1=g$,
 the co-area formula gives the following known equality
 (the intrinsic version of Proposition \ref{propCoArea})
 \begin{equation}
\frac{\partial}{\partial r}\Vol(B^N_{r}) \,
= \, \Vol(\partial B^N_{r}) \quad {\textrm{for
all}}\quad  r \in [\,0, R\,]\,\, .
\end{equation}
so the isoperimetric inequality \ref{eqIntrinIsop} gives
\begin{equation}
\mathcal{G}'(r) \leq 0 \quad ,
\end{equation}
and hence the volume comparison follows.
\end{proof}

%NEW
The intrinsic version of Corollary \ref{corGradientIntBound}
gives likewise the following well-known inequality for the volume of extrinsic spheres, see again \cite{Sa}:
\begin{corollary}
Let $N^n$ denote a complete Riemannian manifold with a pole $p$.
Suppose that the $p$-radial  sectional curvatures of $N^n$ are
bounded from below by the $p_w$-radial sectional curvatures of
a $w$-model space $M^n_w$ for all $r >0$. Then
\begin{equation}
\Vol(\partial D_{r}) \leq\,w^{n-1}(r)=\Vol(\partial B^w_r)
\,\,\,\textrm{for every} \,\,\, r \in [\,0, R\,]
\quad.
\end{equation}
\end{corollary}
%ENDNEW

%%%%%%%%%%%%%%%%%%%%%%%%%%%%%%%%%%%%%%%%%%%%%%%%%%%%%%%%%%%%%%%%%%%%%%%%%%%%%%%%
%%%%%%%%%%%%%%%%%%%%%%%%%%%%%%%%%%%%%%%%%%%%%%%%%%%%%%%%%%%%%%%%%%%%%%%%
%%%%%%%%%%%%%%%%%%%%%%%%%%%%%%%%%%%%%%%%%%%%%%%%%%%%%%%%%%%%%%%%%%%%%%%%%%
%%%%%%%%%%%%%%%%%%%%%%%%%%%%%%%%%%%%%%%%%%%%%%%%%%%%%%%%%%%%%%%%%%%%%%%%%%
%%%%%%%%%%%%%%%%%%%%%%%%%%%%%%%%%%%%%%%%%%%%%%%%%%%%%%%%%%%%%%%%%%%%%%%%%%%%%%%%
%%%%%%%%%%%%%%%%%%%%%%%%%%%%%%%%%%%%%%%%%%%%%%%%%%%%%%%%%%%%%%%%%%%%%%%%
%%%%%%%%%%%%%%%%%%%%%%%%%%%%%%%%%%%%%%%%%%%%%%%%%%%%%%%%%%%%%%%%%%%%%%%%%%
%%%%%%%%%%%%%%%%%%%%%%%%%%%%%%%%%%%%%%%%%%%%%%%%%%%%%%%%%%%%%%%%%%%%%%%%%%

\section{Mean Exit Time Bounds and \\  a Non-explosion Criterion}
\label{secExitTime}

The solution $E$ to the Dirichlet--Poisson problem
\begin{equation}   \label{eqExitDirichletPoisson}
\begin{aligned}
\Delta E &= -1\,\,\,\,\text{on}\,\,\,\, D_{R}\\
E_{|_{\partial D_{R}}} &=0
\end{aligned}
\end{equation}
is, as already alluded to in the proof of Theorem
\ref{thmIsopGeneral}, equation
(\ref{eqMeanExit}), the mean exit time for
Brownian diffusion in $D_{R}$.

A. Gray and M. Pinsky proved in \cite{GraP} a
nice result, which even in the following
truncated version shows in effect, that Brownian
diffusion is fast in negative curvature and slow
in positive curvature, even on the level of
scalar curvature:

\begin{theorem}[Gray and Pinsky]  \label{thmGraP}
Let $B_{r}^{m}(p)$ denote an intrinsic geodesic
ball of small radius $r$ and center $p$ in a
Riemannian manifold $(M^{m}, g)$ which has scalar
curvature $\tau(p)$ at the center point $p$. Then
the mean exit time for Brownian particles
starting at $p$ is
\begin{equation}
E_{r}(p) \,= \, \frac{r^{2}}{2m} +
\frac{\tau(p)\,r^{4}}{12m^{2}(m+2)} +
r^{6}\,\varepsilon(r) \quad ,
\end{equation}
where $\varepsilon(r) \to 0$ when $r \to 0\,$.
\end{theorem}

In view of this reference to the work by Gray and
Pinsky \cite{GraP} we mention here the following
consequence of the proof of Theorem
\ref{thmIsopGeneral}, equations
(\ref{eqMeanExitComp}) and
(\ref{eqMeanExitMonster}).
\begin{theorem} \label{thmExitTime}
Let $\{ N^{n}, P^{m}, C_{w,g,h}^{m} \}$ denote an
isoperimetric com\-pa\-ri\-son constellation which is $w-$balanced from below on
$[\,0, R\,]\,$. Then for all $x\, \in \, D_{R}\, \subset P^{m}$, we have:
\begin{equation} \label{eqExitTime}
E(x) \, \geq \, E_{W}(s(r(x)))\, = \, \int_{s(r(x))}^{s(R)}\,q_{W}(t)
\, dt \quad ,
\end{equation}
where $E_{W}(s)$ is the mean exit time function for Brownian motion in the disc of radius $R$
centered at the pole $p_{W}$ in the model space $C_{w,g,h}^{m}$. The function $q_{W}(s)$ is the isoperimetric quotient
function and $s(r)$ is the stretching function of
the $W-$model comparison space.
\end{theorem}

Moreover, as direct applications of the volume
comparison in Corollary \ref{corVolBound}, using the obvious fact that
the geodesic balls in $P$, $B^P_R$ are subsets of the extrinsic balls $D_R$ of the same radius, we get from the
general non-explosion condition and parabolicity condition in \cite[Theorem 9.1\,]{Gri}:

\begin{proposition}\label{propLogBallCond}
Let $\{ N^{n}, P^{m}, C_{w,g,h}^{m} \}$ denote an
isoperimetric comparison constellation which is $w-$balanced from below on $[\,0,
\infty\,[\,$. Suppose that
\begin{equation}\label{eqLogBallCond}
\int^{\infty}\frac{r(s)\,g(r(s))}{\log(\Vol(B^{W}_{s}))}\,ds\,
= \, \infty \quad.
\end{equation}
Then $P^{m}$ is stochastically complete, i.e. the
Brownian motion in the submanifold is
non-explosive.
\end{proposition}

From the same volume comparison result we get
\begin{proposition} \label{propBallCond}
Let $\{ N^{n}, P^{m}, C_{w,g,h}^{m} \}$ denote an
isoperimetric comparison constellation  which is $w-$balanced from below on $[\,0,
\infty\,[\,$. Suppose that
\begin{equation}\label{eqBallCond}
\int^{\infty}\frac{r(s)\,g(r(s))}{\Vol(B^{W}_{s})}\,ds\,
= \, \infty \quad.
\end{equation}
Then $P^{m}$ is parabolic.
\end{proposition}

We also have the following extrinsic version of
Ichihara's \cite[Theorem 2.1\,]{Ich1}.

\begin{theorem} \label{thmSphereCondGen}
Let $\{ N^{n}, P^{m}, C_{w,g,h}^{m} \}$ denote an
isoperimetric comparison constellation  which is $w-$balanced from below on $[\,0,
\infty \, [\,$. Suppose that
\begin{equation}\label{eqSphereCondGen}
\int^{\infty}\frac{1}{\Vol(\partial
B^{W}_{s})}\,ds\, = \, \infty \quad.
\end{equation}
Then $P^{m}$ is parabolic.
\end{theorem}

\begin{remark} \label{remToFurther}
We shall prove this statement in the next section via a capacity comparison technique, which is
very similar to the isoperimetric comparison developed in the previous sections. One first pertinent remark
is that Theorem \ref{thmSphereCondGen} is {\em{not}}
a consequence of the previous proposition \ref{propBallCond} as pointed out by
A. Grigor{\cprime}yan in \cite[p. 180--181]{Gri}. Indeed, he shows
by a nice example, that \ref{eqBallCond} is not equivalent to parabolicity
of the model space, so, although \ref{eqSphereCondGen} does imply parabolicity
of the model space, we cannot conclude parabolicity of $P^{m}$
via Proposition \ref{propBallCond}.
\end{remark}

As a corollary to Theorem \ref{thmSphereCondGen} we get Theorem \ref{thmAparab} as follows:
\begin{proof}[Proof of Theorem \ref{thmAparab}]
The strong conditions for the theorem as stated in the introduction, i.e. $h(r)\,=\,C$ and $w(r) \, = \, Q_{b}(r)$,
$b \, \leq \, 0$, still implies that the model space in the comparison constellation, $\,C_{Q_{b},g,C}^{m}\,\,$, is $Q_{b}-$balanced from below. Furthermore, since we have that $1 \, \geq \, g(r) \, > \, 0$,
the 'stretching' satisfies $s(r) \to \infty$ when $r \to \infty$.
We have then according to Proposition \ref{propMonster}, (\ref{eqLambdaNew}):
\begin{equation}
\begin{aligned}
&\int^{\infty}\frac{1}{\Vol(\partial
B^{W}_{s})}\,ds\, \\
&= \, \int^{\infty}\frac{1}{W^{m-1}(s(r))}\,\left(\frac{ds}{dr}\right)\,dr \, \\
&= \,  \int^{\infty}\frac{1}{\Lambda(r)\,g(r)}\,dr \, \\
&= \,  \int^{\infty} \left( \frac{T}{Q_{b}(r)}\,\exp\left(-\int_{r}^{\,1}\,\frac{m}{g^{2}(t)}\,\left(\eta_{Q_{b}}(t)
- C\right)\, dt\right)\right)^{-1}\,dr,
\end{aligned}
\end{equation}
where $T$ is the fixed constant defined and found in Proposition \ref{propMonster}.
In this particular case the condition (\ref{eqSphereCondGen}) therefore corresponds to (\ref{eqAparab})
in Theorem \ref{thmAparab}, which is thence proved.
\end{proof}

The intrinsic versions of these results (implying formally that
$h(r) \, = \, 0$ and $g(r) \,= \, 1$ for all $r$) are well
known and established by Ahlfors, Nevanlinna,
Karp, Varopoulos, Lyons and Sullivan, and
Grigor'yan, see \cite[Theorem 7.3 and Theorem
7.5\,]{Gri}. For warped product model space manifolds, the reciprocal
boundary-volume integral condition is also
necessary for parabolicity, but {\em{in general}} none
of the conditions are necessary.

%%%%%%%%%%%%%%%%%%%%%%%%%%%%%%%%%%%%%%%%%%%%%%%%%%%%%%%%%%%%%%%%%%%%%%%%%%%%%%%%
%%%%%%%%%%%%%%%%%%%%%%%%%%%%%%%%%%%%%%%%%%%%%%%%%%%%%%%%%%%%%%%%%%%%%%%%
%%%%%%%%%%%%%%%%%%%%%%%%%%%%%%%%%%%%%%%%%%%%%%%%%%%%%%%%%%%%%%%%%%%%%%%%%%
%%%%%%%%%%%%%%%%%%%%%%%%%%%%%%%%%%%%%%%%%%%%%%%%%%%%%%%%%%%%%%%%%%%%%%%%%%
%%%%%%%%%%%%%%%%%%%%%%%%%%%%%%%%%%%%%%%%%%%%%%%%%%%%%%%%%%%%%%%%%%%%%%%%%%%%%%%%
%%%%%%%%%%%%%%%%%%%%%%%%%%%%%%%%%%%%%%%%%%%%%%%%%%%%%%%%%%%%%%%%%%%%%%%%
%%%%%%%%%%%%%%%%%%%%%%%%%%%%%%%%%%%%%%%%%%%%%%%%%%%%%%%%%%%%%%%%%%%%%%%%%%
%%%%%%%%%%%%%%%%%%%%%%%%%%%%%%%%%%%%%%%%%%%%%%%%%%%%%%%%%%%%%%%%%%%%%%%%%%

\section{Capacity Analysis}
\label{secCapAnalysis}

Given the extrinsic balls with radii $\rho <R$, $D_{\rho}$ and $D_R$, the annulus $A_{\rho,R}$ is defined as

$$A_{\rho,R}=D_R \sim D_{\rho} $$

The
unit normal vector field on the boundary of this annulus $\partial A_{\rho,R} =
\partial D_{\rho} \cup \partial D_R$ is denoted by $\nu$ and defined by the
following normalized $P$-gradient of the distance function restricted to
$\partial D_{\rho}$ and $\partial D_R$, respectively:
$$
\nu =
 \nabla^P r(x)/ \Vert\nabla^{P} r(x) \Vert\,\,\, , \quad x \in \partial A_{\rho,R} \quad  .
$$

We apply the previously mentioned result by Greene and Wu, \cite{GreW}, with a lower radial curvature
bound again, but now we consider radial functions
with $f'(r) \, \geq \, 0$ which will change the
inequality in the Laplace comparison, equation
(\ref{eqLap1}).

%%%%%%%%%%%%%%%%%%%%%%%%%%%%%%%%%%%%%%%%%%%%%%%%%%%%%%%%%%%%%%%%%%%%%%%%%%

\begin{theorem} \label{thmCapComp}
Let $\{ N^{n}, P^{m}, C_{w,g,h}^{m} \}$ denote an
isoperimetric comparison constellation  which is $w-$balanced from below on $[\,0,
\infty \, [\,$.  Then
 \begin{equation}\label{eqCapacGeneral}
    \begin{aligned} \C(A_{\rho, R}) \, \leq
    \,  \left(\int_{s(\rho)}^{s(R)}\,\frac{1}{W^{m-1}(t)}
\,dt \right)^{-1}  \quad .
    \end{aligned}
    \end{equation}
\end{theorem}

\begin{proof}
We consider again the second order differential
operator $\LL$ but now we look for the smooth
solution $\xi(r)$ to the following
Dirichlet--Laplace problem on the interval $[\,
\rho, R\,]$, $\rho > 0$:
\begin{equation}   \label{eqLaplace1}
\begin{aligned}
\LL \xi(r) &= 0\,\,\,\,\text{on}\,\,\,\, [\,\rho, R\,] \quad ,\\
\xi(\rho) \, &=\, 0 \,\,\, , \,\,\, \xi(R) \, =
\, 1 \quad .
\end{aligned}
\end{equation}
The solution is again constructed via the
function $\Lambda(r)$ defined in equation
(\ref{eqLambdaNew}):
\begin{equation}
\xi'(r) \, = \, \Xi(r) \, = \,
\frac{1}{g(r)\Lambda(r)}\,\left(\int_{\rho}^{R}\,\frac{1}{g(t)\Lambda(t)}
\,dt \right)^{-1} \quad ,
\end{equation}

Then we have
\begin{equation}
\xi(r) \,= \,
\left(\int_{\rho}^{r}\,\frac{1}{g(t)\Lambda(t)}\,dt\right)\,\left(\int_{\rho}^{R}\,\frac{1}{g(t)\Lambda(t)}
\,dt \right)^{-1} \quad .
\end{equation}

Applying the Laplace inequality (\ref{eqLap1}) on
the radial functions $\xi(r)$ - now with a
non-negative derivative - transplanted into
$P^{m}$ in $N^{n}$ now gives the following
comparison inequality, using the assumptions
stated in the theorem:
\begin{equation}
\begin{aligned}
\Delta^{P}\xi(r(x)) \, &\leq \,
\left(\xi''(r(x))- \xi'(r(x))\eta_{w}(r(x))\right)\,g^{2}(r(x)) \\
& \phantom{mm}+ m\xi'(r(x))\left( \eta_{w}(r(x))
- h(r(x))\right)\, \\ &= \, \LL \xi(r(x)) \, \\
&= \, 0 \, \\ &= \, \Delta^{P}v(x) \quad ,
\end{aligned}
\end{equation}
where $v(x)$ is the Laplace potential function
for the extrinsic annulus  $A_{\rho, R}\, = \,
D_{R} - D_{\rho}$, setting $v_{\partial D_{\rho}}\,=\,0$ and
$v_{\partial D_{R}}\,=\,1$.\\

For this inequality to hold we need that
\begin{equation}
\xi''(r)- \xi'(r)\eta_{w}(r) \, \leq \, 0 \quad .
\end{equation}

This follows from the Laplace equation itself
together with the consequence (\ref{eqEtaVSh}) of the balance condition that $h(r) \, \leq\,
\eta_{w}(r)$:
\begin{equation}
\left(\xi''(r)-
\xi'(r)\eta_{w}(r)\right)\,g^{2}(r) \, = \,
-m\xi'(r)\left( \eta_{w}(r) - h(r) \right) \,
\leq 0 \quad .
\end{equation}

The maximum principle then applies again and
gives:
\begin{equation}
v(x) \, \leq \, \xi(r(x))\,\, , \,\, \textrm{for
all} \, \, x \in A_{\rho, R} \quad .
\end{equation}

This implies in particular that at $\partial
D_{\rho}$ we have
\begin{equation}
\Xi(\rho) \, = \, |\xi'(\rho) | \, \geq \, \Vert
\nabla^{P}v(x)_{|_{\partial D_{\rho}}} \Vert
\end{equation}

Applying the divergence theorem and using
Corollary \ref{corGradientIntBound}, we get
    \begin{equation}\label{eqCapac2}
    \begin{aligned} \C(A_{\rho, R}) &= \, \int_{\partial D_{\rho}}\Vert \nabla^{P}v(x) \Vert \, d\nu \\ &\leq
\, \Xi(\rho)\int_{\partial D_{\rho}}\Vert
\nabla^{P}r \Vert \, d\nu \\ &\leq \,
\Xi(\rho)\,g(\rho)\,\Lambda(\rho)
\\ &=
    \,  \left(\int_{\rho}^{R}\,\frac{1}{g(t)\Lambda(t)}
\,dt \right)^{-1} \\
&=\, \left(\int_{s(\rho)}^{s(R)}\,\frac{1}{W^{m-1}(t)}
\,dt \right)^{-1}
 \quad ,
    \end{aligned}
    \end{equation}
    which shows the general capacity bound (\ref{eqCapacGeneral}).

\end{proof}

%%%%%%%%%%%%%%%%%%%%%%%%%%%%%%%%%%%%%%%%%%%%%%%%%%%%%%%%%%%%%%%%%%%%%%%%
%%%%%%%%%%%%%%%%%%%%%%%%%%%%%%%%%%%%%%%%%%%%%%%%%%%%%%%%%%%%%%%%%%%%%%%%%%%%

We then finally have the following promised consequence alluded to in the previous section:

\begin{corollary}[Theorem \ref{thmSphereCondGen}]
Under the conditions of Theorem \ref{thmCapComp}:
If furthermore
\begin{equation} \label{eqWcond}
\int^{\,\infty}\,\frac{1}{W^{m-1}(t)} \,dt \, =
\, \infty \quad ,
\end{equation}
then $P^{m}$ is parabolic.
\end{corollary}

\begin{proof}
Referring to the above capacity inequality (\ref{eqCapacGeneral}),
we see that the capacity is forced to $0$ in the
limit $s \, \to \, \infty$ because of the condition (\ref{eqWcond}).
According to the Kelvin--Nevanlinna--Royden criterion (b) of the introduction,
see subsection \ref{subsecGlobal},
this corresponds to parabolicity of the
submanifold.
\end{proof}

%%%%%%%%%%%%%%%%%%%%%%%%%%%%%%%%%%%%%%%%%%%%%%%%%%%%%%%%%%%%%%%%%%%%%%%%%%

%%%%%%%%%%%%%%%%%%%%%%%%%%%%%%%%%%%%%%%%%
%%%%%%%%%%%%%%%%%%%%%%%%%%%%%%%%%%%%

\bibliographystyle{acm}

%\bibliography{ConvexLIB}

\def\cprime{$'$}

\end{document}